\documentclass[journal]{IEEEtran}
\usepackage[utf8]{inputenc}
\usepackage{mathtools}
\mathtoolsset{showonlyrefs,showmanualtags}
\usepackage{amssymb,amsthm}
\usepackage{dsfont}
\usepackage{color}
\usepackage{bbold} 
\usepackage{booktabs}
\usepackage[yyyymmdd,hhmmss]{datetime}
\usepackage{bm,upgreek}
\usepackage{tablefootnote}
\usepackage{nth}

 \usepackage{tikz}
\usetikzlibrary{graphs,graphs.standard}
\usetikzlibrary{shapes,arrows}\usetikzlibrary{arrows,decorations.pathmorphing,backgrounds,positioning,fit,petri}

\usepackage[ruled]{algorithm} 
\usepackage{algpseudocode}

\newtheorem{problem}{Problem}
\newtheorem{assumption}{Assumption}
\newtheorem{theorem}{Theorem}
\newtheorem{definition}{Definition}
\newtheorem{remark}{Remark}
\newtheorem{lemma}{Lemma}
\newtheorem{corollary}{Corollary}
\newtheorem{proposition}{Proposition}

\usepackage[backend=biber,style=ieee,citestyle=numeric-comp]{biblatex}

\newcommand{\Ninf}[1]{ \| #1 \| }
\newcommand{\Exp}[1]{\mathbb{E}[ #1]}
\newcommand{\Prob}[1]{\mathbb{P}(#1)} 
\newcommand{\ID}[1]{ \mathbb{1} (#1 ) }
\newcommand{\B}{ \mathtt{blank} }
\newcommand{\N}{ \mathbb{N}^\ast }
\newcommand{\TM}{T_{\scriptstyle \mathsf{m}}}
\newcommand{\argmin}{\operatornamewithlimits{argmin}} 
\newcommand{\argmax}{\operatornamewithlimits{argmax}} 
\DeclareMathOperator{\Binopdf}{binopdf}
\DeclareMathOperator{\TR}{Tr}
\DeclareMathOperator{\DIAG}{diag}
\DeclareMathOperator{\VEC}{vec}
\newcommand{\Mspace}{\mathcal{M}(n)}

\newcommand{\Compress}{\medmuskip=0mu
\thinmuskip=0mu
\thickmuskip=0mu}

\tikzset{
    >=stealth',
    punkt/.style={
           rectangle,
           rounded corners,
           draw=black, very thick,
           text width=4.5em,
           minimum height=2em,
           text centered},
}

\title{Data Collection versus Data Estimation: A Fundamental Trade-off in Dynamic  Networks
}

\author{Jalal~Arabneydi,~\IEEEmembership{Member,~IEEE,}
        and~Amir~G.~Aghdam,~\IEEEmembership{Senior Member,~IEEE}
\IEEEcompsocitemizethanks{\IEEEcompsocthanksitem This work has been supported in part by Natural Sciences and Engineering  Research Council of Canada (NSERC) under Grant RGPIN-262127-17, and in part by Concordia University under Horizon Postdoctoral Fellowship.
\IEEEcompsocthanksitem The authors are with Department of Electrical and Computer Engineering, 
        Concordia University, 1455 de Maisonneuve Blvd, Montreal, QC, Canada.
{\tt\small Email:jalal.arabneydi@mail.mcgill.ca} and
        {\tt\small Email:aghdam@ece.concordia.ca}
 \IEEEcompsocthanksitem This article has supplementary downloadable material available at {\tt\small http://ieeexplore.ieee.org}, provided by the authors.
  \IEEEcompsocthanksitem Digital Object Identifier 10.1109/TSNE.2020.2966504
}
}
\begin{document}
\maketitle
\thispagestyle{empty}
\pagestyle{empty}
\begin{abstract}
An important   question that often arises in the operation of networked systems  is whether to collect the real-time  data  or to estimate them based on the previously collected data. Various factors should be taken into account  such as  how informative the data are at each time instant  for state estimation, how costly and credible the collected  data are, and how rapidly the data vary with time.  The above question can be  formulated as  a dynamic decision making problem with imperfect information structure, where  a decision maker wishes to find an efficient way to switch between data collection and data estimation while the quality of the estimation depends on the  previously collected data (i.e., duality effect). In this paper, the  evolution  of the state of each node is  modeled as   an exchangeable Markov process for discrete features and equivariant linear system for continuous features, where  the data  of interest are  defined in the former case  as the empirical distribution  of the states,  and in the latter case  as the weighted average of the states.  When the data are collected, they may or may not be credible, according to  a Bernoulli distribution.    Based on a novel planning space, a Bellman equation is proposed to identify a near-optimal  strategy whose computational complexity is  logarithmic  with respect to the inverse of the  desired  maximum distance from the optimal solution, and  polynomial  with respect to the number of  nodes.  A reinforcement learning algorithm is developed for the case when  the model is not known exactly, and its convergence to the near-optimal solution is shown subsequently.  In addition, a certainty threshold is introduced  that determines when data estimation is more desirable than data collection, as the number of nodes increases. For the special case of linear dynamics, a separation principle is constructed wherein the optimal estimate is computed by a Kalman-like filter, irrespective of the probability distribution of random variables. It is shown that the
complexity of finding the proposed sampling  strategy, in this special case,  is independent of the size of the state space and the number of nodes.  Examples of a sensor network, a communication network and a social network   are provided.
\end{abstract}

\section{Introduction}
The trade-off between the cost  and value of information  emerges in different types of networks, where  it is important to monitor the status of the network by analyzing its real-time data.   There is often  a cost associated with collecting  the  data; hence, the question arises as    when an estimate of the network  data would be more desirable than monitoring the operation of the network.

A  \emph{partially observable Markov decision process} (POMDP) model is presented in this paper  to address the above  trade-off in invariant/equivariant  networks.  The data are defined as  the empirical distribution and weighted average of the states of the nodes   for finite and infinite state spaces, respectively, and   the  per-step cost function incorporates  the  two  competing  concepts of the above trade-off:  the cost of information and  the value  of  information. It is to be noted that the empirical distribution  and  weighted average (i.e., linear regression) of  states are the data  of interest in  various  applications, specially the ones modelled by deep neural networks  that  have recently received much attention~\cite{lecun2015deep,schmidhuber2015deep}. 
For example,  the empirical distribution appears in  Markov-chain deep teams  wherein  the nodes are  partitioned into several sub-populations such that  the system is invariant to  the permutation  of  nodes in each sub-population~\cite{Jalal2019MFT}.  The  weighted average, on the other hand,  emerges  in linear quadratic  deep teams  wherein  the nodes in each sub-population are not necessarily exchangeable and can have different weights~\cite{Jalal2019risk}. See other examples in  social networks~\cite{jadbabaie2012non}, epidemics~\cite{van2009virus}, cyber-security networks~\cite{wang2014mean} and smart grids~\cite{JalalACC2018}, to name only a few.

 In general,  finding  the optimal solution of an infinite-horizon discounted cost  POMDP  is  an undecidable problem~\cite{madani2003undecidability},\footnote{The computational complexity of  finding  the  optimal solution  of  the finite-horizon POMDP  is PSpace-complete,  whereas that of MDP is P-complete~\cite{papadimitriou1987complexity}.}  
and even  finding a near-optimal solution  is  NP-hard~\cite{meuleau1999solving}. On the other hand,  it is not always possible  to  have sufficiently accurate knowledge of the model of nodes.  Thus,  it is  important  to  learn  the solution  when the exact  model is not available, and this clearly adds to the complexity of the problem.  Due to the complexity of finding a near-optimal solution,  most of the existing work is mainly  focused on developing  approximation methods for  POMDPs, some of which  are briefly  reviewed in the  next paragraph.

 In grid-based approaches, an approximate value function is computed at a fixed number of points (called grids), and then interpolated over the entire belief space~\cite{lovejoy1991computationally}. The advantage of such approaches is that their computational complexity remains the same at every iteration (i.e., does not increase with time) but their drawback is that the fixed points may not be reachable. In point-based methods, the reachability problem is addressed by considering the reachable set only, where  an approximate value function is calculated iteratively in terms of  $\alpha$-vectors\footnote{The notion of $\alpha$-vectors was  introduced in~\cite{smallwood1973optimal} to solve the finite-horizon POMDP, and was later enhanced  % to ``witness'' algorithm
  by pruning  dominated vectors~\cite{littman1994witness,kaelbling1998planning}.
 }  over a finite number of points in the reachable set~\cite{Shani2013survey}.
Note that unlike grid-based approaches,  here the points are not fixed and may change as the value function changes. In policy-search methods such  as finite-state controllers, attention is restricted to a certain class (structure) of strategies, and the objective is to find the best strategy in that class using policy iteration and gradient-based techniques~\cite{hansen1998finite}.  The above methods  often use  the  belief space as the planning space, whose size  grows exponentially with the number of nodes and  time horizon; thus, they suffer from the curse of dimensionality. In addition, the dynamics of the belief state  depends on the transition probability matrix  (i.e., it is a model-based state). Hence, it is not clear how the above POMDP solvers  may  be used  in the case where the model is not exactly known~\cite{krishnamurthy2016partially}.   For more details on POMDPs, the interested reader is referred to~\cite{Cassandra1998survey,krishnamurthy2016partially} and references therein.

In this paper,  it is desired to find a near-optimal strategy that determines, at each time instant,   whether to collect the data or to estimate them.  In  contrast to   POMDP solvers mentioned in the previous paragraph,  we  exploit the structure of the problem  to efficiently solve the resultant POMDP by  using a different planning space (which is smaller  than the belief space and is  independent of the model).\footnote{The idea of defining a pseudo-state  to  plan and learn in POMDPs by taking  the structure of the problem into account was first  introduced in~\cite[Chapters 5,6]{arabneydi2016new} and a preliminary result was presented in~\cite{JalalACC2015}.  In particular,  the pseudo-state (called incrementally expansion representation) provides a guaranteed optimality bound, and is  more general than  (model-based) belief state and (model-free) history state.} To determine the complexity of the above problem,  it is noted that  the  information structure is imperfect,  and  the  model structure is not necessarily complete (known), yet the decision maker  wishes to  sequentially choose the frequency of  collecting data based on several factors. Some of such factors include, for example,  how informative  the data are at  each time instant to be fed to the estimator,   or,  how costly and credible that data are  upon collection, and how quickly the states of nodes change with time.  Three  examples are provided to illustrate the efficacy of   the proposed strategy. The first example studies  a  sensor network, where  a sensor measures a  dynamic phenomenon with saturation levels such as  the temperature of a room,   and  reports it to a data center over an unreliable link at some communication cost  in such a way that  the   estimate   constructed at  the data center is reliable. The second example studies a communication network with two different topologies, where the weighted average of an arbitrary number of Markov processes with linear dynamics is to be collected at some transmission cost.  The third example deals with a social network, where an agency is to conduct  a survey  from  a number of users by offering monetary rewards  as an  incentive for participation.

% However, Sections V, VI and VII along with
%Theorems 6--10 and Examples 1--3 in this paper are new.

%\edit{When the model is not known, designing reinforcement learning algorithm is difficult because most of the reinforcement algorithms are designed for finite state-action MDPs and do not work for POMDPs. For example, the standard Q-learning algorithm is not implementable in POMDPs for two reasons: (a) belief space is uncountable and (b) knowing the belief state itself requires the knowledge of the model. }

The paper is organized as follows. In Section~\ref{sec:problem},   the  problem is  formulated  for both known and unknown models.   To  facilitate the analysis,   some preliminary results on the  dynamics  of the empirical distribution are derived  in Section~\ref{sec:preliminary}. The proposed solutions  are then presented  in Section~\ref{sec:main}  followed by   the asymptotic analysis  in Section~\ref{sec:asymptotic}.  The special case of   linear dynamics  is  studied  in Section~\ref{sec:linearcase}.   The extension of the obtained results to more complex networks is discussed in Section~\ref{sec:gen}.  To  demonstrate the efficacy of the proposed strategies,  three numerical examples are provided in Section~\ref{sec:applications}.  Finally,  some concluding remarks are given  in Section~\ref{sec:conclusions}. 

\subsection{Notation}
Throughout this paper,   $\mathbb{N}$, $\mathbb{Z}$ and $\mathbb{R}$  denote the set of natural, integer and real numbers, respectively.  For any $k \in \mathbb{N}$,   the shorthand notations   $\mathbb{N}_k$, $\mathbb{N}^\ast_k$ and $\mathbb{Z}_k$  are used to represent the finite set of integers  $\{1, \ldots, k\}$, $\{0, 1, \ldots, k\}$ and $\{-k,\ldots, 0, \ldots, k\}$, respectively.  Given vectors $a,b,c$,  $\VEC(a,b,c)=[a^\intercal, b^\intercal, c^\intercal]^\intercal$. For any $a,b \in \mathbb{N}$,  $a \leq b$,    $x_{a:b}$ is defined as the vector $(x_a,x_{a+1},\ldots,x_{b})$. Moreover, $\ID{\boldsymbol \cdot}$ is the indicator function, $\Prob{\boldsymbol \cdot}$ is the probability of an event,  $\Exp{\boldsymbol \cdot}$ is the expectation of a random variable,  $\Ninf{\boldsymbol \cdot}$  is the infinity norm of a vector,   $\TR(\boldsymbol \cdot)$ is the  trace of a matrix, and $\DIAG(\boldsymbol \cdot)$ is a diagonal matrix.   Given a finite  set $\mathcal{S}$, $\mathcal{P}(\mathcal{S})$ denotes the space of probability measures over set $\mathcal{S}$, and $|\mathcal{S}|$ is the size of this finite set. The notation $\Binopdf(n,p)$  is used for   the binomial probability distribution function with $n \in \mathbb{N}$ trials,   and success probability $p \in [0,1]$.  For any square matrix $A$,  $A^0$ is  the identity matrix of the same size as $A$.

\subsection{Main contributions}
  It is generally  difficult to find the optimal control strategy when the information structure  is imperfect   because of  the \emph{dual effect} phenomenon~\cite{bar1974dual}, where the control strategy affects the estimation process by altering  the observation dynamics  while the   estimation strategy  impacts  the control strategy by determining the state  estimate. As a result,   it is of particular interest in control  theory  to identify  decision-making problems with imperfect information structure that admit tractable solutions. The main contributions of this paper are outlined below.
 \begin{enumerate}
 \item   A theoretical analysis is provided  for an important trade-off that emerges in different forms in networked systems, namely, collecting data versus estimating data. In particular, we exploit the structure of  the problem to introduce a new planning space, which  is different from the conventional  belief space (Theorem~\ref{proposition:bellman}).  The salient feature of the proposed space is that it leads to a low-complexity near-optimal solution  when the model is known (Theorem~\ref{thm:approximate-pomdp}). Furthermore,  the solution methodology  can be extended to the case of systems with unknown models while such an extension in the belief space is conceptually difficult  because the belief space is model-dependent. More precisely, we develop  reinforcement learning algorithms  to deal with  unknown   probability distributions  and unmodelled dynamics (Proposition~\ref{thm:machine} and Theorem~\ref{thm:RL}). 
 \item  To provide a more practical setup, we define  a Bernoulli probability distribution  to design a robust strategy in order to  embody  data uncertainty.  We show that our results naturally extend to the case where  observations are received with a constant time delay. We also study the role of the number of nodes, and  define a certainty threshold  that determines when data estimation becomes more desirable than  data collection, as the number of nodes increases (Theorem~\ref{thm:asymptotic} and Corollary~\ref{cor:certainty}). 
 % When the number of nodes is large, an infinite-population scale-free estimator is defined whose performace is shown to be stable under a sufficient condition. 
 \item  For the special case of linear dynamic systems,  we   establish a separation principle between control and estimation (that also holds for non-Gaussian random variables), which is different  from the existing result for  only Gaussian random variables (Theorem~\ref{thm:separation}).  We then provide an efficient dynamic-program-based solution for computing a near-optimal scheduling strategy, where  the proposed planning space   requires no knowledge of the underlying probability distributions  and  corresponding matrices (Theorem~\ref{cor:linear}).    It is to be noted that our proposed  strategy is nonlinear even in the case of  linear dynamics  with quadratic cost function. 
 \item We show that the computational complexity of the proposed strategy is  linear   with respect to the approximation index,  to be defined later,  and  polynomial  with respect to the number of  nodes.  For the special case of  linear quadratic cost functions, the complexity of computing the strategy is independent of the number of nodes as well as the  dimension of the  state space.
 
 \item We generalize the obtained results to the following cases: (a) multiple decision makers and multiple estimators; (b) multiple reset actions; (c) partially exchangeable and partially equivariant  networks, and (d) Markovian noise and Markovian credibility processes (see Section~\ref{sec:gen} for details).
 \end{enumerate}

\section{Problem Formulation} \label{sec:problem}

%\edit{\subsection{Model}}

Consider an index-invariant (exchangeable) network of $n \in \mathbb{N}$  nodes. Let $s^i_t \in \mathcal{S}$ denote the state of node $i \in \mathbb{N}_n$ at time $t \in \mathbb{N}$,  where $\mathcal{S} \subset \mathbb{R}$ is a known finite set. Denote  by $m_t \in \Mspace$ the empirical distribution of  states  at time $t$, i.e.,
\begin{equation}\label{eq:mf-def}
m_t(s)=\frac{1}{n}\sum_{i=1}^n \ID{s^i_t=s}, \quad s \in \mathcal{S},
\end{equation}
where $\Mspace=\{(\alpha_1,\ldots,\alpha_{|\mathcal{S}|}) \big| \alpha_i \in \{0, \frac{1}{n},\ldots,1\}, i \in \mathbb{N}_{|\mathcal{S}|}, \sum_{i=1}^{|\mathcal{S}|} \alpha_i=1\}$  is the set of empirical distributions  over the state space $\mathcal{S}$ with $n$ samples.
In the sequel, the empirical distribution of states is sometimes  referred to as data.
 It is shown in~\cite{arabneydi2016new} that any set of  exchangeable Markov processes $\mathbf s_t=\VEC(s^1_t,\ldots,s^n_t)$   can be equivalently expressed as a set of  Markov processes  coupled  through the empirical distribution of states.   Therefore,  let the state of   the $i$-th node  at  time $t$  evolve according to the following dynamics:
\begin{equation}\label{eq:model-dynamics}
s^i_{t+1}=f(s^i_t,m_t, w^i_t),\quad i \in \mathbb{N}_n, t \in \mathbb{N},
%\mathbf s_{t+1}=f(\mathbf s_t,\mathbf w_t),
\end{equation}
where $w^i_t \in \mathcal{W} \subset \mathbb{R}$ is the local noise of node $i $ at time $t$.    In addition, consider a decision maker that wishes to find an affordable way  to  sample the  data over time horizon  such that  the  estimate  of the data, constructed based on the previously sampled  data,  is reliable.    Let $a_t \in \mathcal{A}:=\{0,1\}$  denote the action of  the decision maker at time $t \in \mathbb{N}$, where  $a_t=1$ means that the decision maker collects the data and $a_t=0$  means that  it does not  collect them. 

 In real-world applications, it is possible that the collected data  are not credible  due to, for instance,  misinformation induced by fake news in social networks, packet drop  in communication networks, and faulty encoders and decoders in sensor  networks.   Denote by $q \in [0,1]$   the probability that  the collected  data are credible (correct). When the data are not credible, they are discarded.  Note that the evaluation of the  credibility  of the data may be viewed as  an exogenous process that can have any arbitrary Markov-chain dynamics.   This extension does not add much complexity to our analysis, and  is not considered here for simplicity of notation.

 Denote by $o_t \in \mathcal{O}:= \Mspace\cup \{\B\}$ the observation of the decision maker at time $t$, where $\B$ implies that the decision maker receives either no observation, when $a_t=0$, or potentially incorrect (unreliable) data, when $a_t=1$. Subsequently,  for  any $t \in \mathbb{N}$ and  $m \in \Mspace$, 
\begin{equation}\label{eq:observation_blank}
\Prob{o_{t+1}=\B|m_{t+1}=m,a_t=0}=1, 
\end{equation}
and 
\begin{align}\label{eq:observation_not_blank}
\Prob{o_{t+1}=m|m_{t+1}=m,a_t=1}&=q,\nonumber \\
\Prob{o_{t+1}=\B|m_{t+1}=m,a_t=1}&= 1-q, 
\end{align}
where initially $o_1=m_1$. 
 The decision maker determines its action at time $t$ based on its information by that time, i.e., 
\begin{equation} \label{eq:strategy}
a_t=g_t(o_{1:t},a_{1:t-1}),
\end{equation}
where $g_t:\mathcal{O}^t \times \mathcal{A}^{t-1}\rightarrow \mathcal{A}$ is called the control law of the decision maker. Denote by $g:=\{g_1,g_2,\ldots\}$  the strategy of the decision maker.  Note that the information structure of the decision maker  is imperfect because it only has access to the  collected  data.

 Let $\mathbf w_t=\VEC(w^1_t,\ldots,w^n_t) \in \mathcal{W}^n$, $ t \in \mathbb{N}$, with a probability distribution function $P_{\mathbf w}$. 
Let  also $\eta_t \in \{0,1\}$  denote the  random variable representing the credibility process  at time $t$ such that $\Prob{\eta_t=1}=~q$. It is assumed that the primitive random variables $\{\mathbf s_1,\mathbf w_1,\mathbf w_2,\ldots, \eta_1,\eta_2, \ldots\}$ are defined on a common probability space,  are mutually independent, and have finite variances.   At any time $t \in \mathbb{N}$,  the decision maker constructs an estimate of  the data $m_t$, denoted by $\hat m_t \in  \Mspace$,  according to an estimator function $h$  as follows:
\begin{equation}\label{eq:estimator:general}
\hat m_t=h(\Prob{m_t|o_{1:t}, a_{1:t-1}}).
\end{equation}
 The objective of the decision maker  is to design a strategy  that not only  keeps the estimation error small but also incurs  minimal collection cost.  To this end, we define the  following per-step cost $c: \Mspace^2 \times \mathcal{A} \rightarrow \mathbb{R}_{\geq 0}$, i.e.
\begin{equation}\label{eq:per_step_cost_general}
c(m_t,\hat m_t, a_t),
\end{equation}
where $c(m_t,\hat m_t, 0)$ is the cost when the decision maker chooses not to  collect  data at time $t$, and  $c(m_t,\hat m_t, 1)$ is the cost when it  is decided to  collect  data.

\subsection{Linear dynamics}\label{sec:linear_model}
Since the role of  topology is implicitly described in~\eqref{eq:model-dynamics},  we present an equivariant linear  network in this subsection  whose topology can be  described explicitly  in terms of the  eigenvalues and eigenvectors of its adjacency matrix.  In particular, consider an undirected   weighted graph  with  a real-valued  symmetric adjacency matrix $\mathbf A$. One standard way to vectorize the graph is to use  the spectral decomposition~\cite{chung1997spectral}, i.e., 
\begin{equation}
 \mathbf A \approx [\sum_{l=0}^L \alpha(l) (\mathbf A)^l]= [\sum_{l=0}^L   \alpha(l) (  \sum_{j=1}^n \lambda_j \mathbf V(:,j) \mathbf V(:,j)^\intercal)^l ],
\end{equation}
where  $\alpha(l) \in \mathbb{R}$, $\mathbf V$ is an orthogonal matrix  consisting of the eigenvectors of $\mathbf A$, $\Lambda$ is a diagonal matrix of the corresponding eigenvalues, and $\mathbf A^l= \mathbf V \Lambda^l \mathbf{V}^\intercal= \sum_{j=1}^n \lambda_j^l \mathbf V(:,j) \mathbf V(:,j)^\intercal$, $ l \in \mathbb{N}_L$. Let $D \ll n $ be  the number of ``dominant'' eigenvalues.  Then, the dynamics of the network can be expressed as follows:
\begin{align}
\mathbf s_{t+1}&= \mathbf A \mathbf s_t \approx [\sum_{l=0}^L   \alpha(l) (  \sum_{d=1}^D \lambda_d \mathbf V(:,d) \mathbf V(:,d)^\intercal)^l ]\mathbf s_t\\
&=  \sum_{d=1}^D A_d \mathbf V(:,d)  m^d_t,
\end{align}
where  $m^d_t:= \mathbf V(:,d)^\intercal  \mathbf s_t=\sum_{i=1}^n \mathbf V(i,d) s^i_t$ and  $A_d:=\sum_{l=0}^L   \alpha(l) (\lambda_d)^l$, $d \in \mathbb{N}_D$. Hence, for any mode  $d \in \mathbb{N}_D$:
\begin{equation}\label{eq:dynamics_example_3}
m^d_{t+1}= \mathbf V(:,d)^\intercal  \mathbf s_{t+1} \approx  A_d  m^d_t,
\end{equation}
where  $\mathbf V(:,d)^\intercal \mathbf V(:,j)$ is equal to the Kronecker delta function $\delta_{d,j}$ for any  $d,j \in \mathbb{N}_n$. Based on the above vectorized representation, consider now a network of $n \in \mathbb{N}$ nodes, wherein the state of node $i \in \mathbb{N}_n$ at time $t \in \mathbb{N}$ is a vector  denoted by $s^i_t \in \mathbb{R}^{d_s}$.  Let  $m^d_t$ be  the weighted average of the states at time~$t$ associated with  the $d$-th dominant mode, i.e.
\begin{equation}
m^d_t:=\frac{1}{n}\sum_{i=1}^n v^{i,d}s^i_t,\quad   \bar w^d_t= \frac{1}{n}\sum_{i=1}^n v^{i,d}w^i_t,
\end{equation}
 where  $d$-th eigenvector  is normalized as  $\frac{1}{n}\sum_{i=1}^n (v^{i,d})^2=~1$,  and  $w^i_{1:\infty}$ is a random process  with zero mean and finite covariance matrix with a known bound $\Sigma^{i,w}  \leq \Sigma_{\text{max}} \in \mathbb{R}^{d_s \times d_s}$. Let the dynamics of the augmented weighted average be described by the following linear equation:
\begin{equation}\label{eq:dynamic_example_3}
m_{t+1}= A m_t+ \bar w_t,
\end{equation}
where $m_t=\VEC(m^1_t,\ldots,m^D_t)$, $A=\DIAG(A_1,\ldots,A_D)$, and $\bar w_t=\VEC(\bar w^1_t,\ldots,\bar w^D_t)$. The per-step cost function at time $t \in \mathbb{N}$ is defined as:
\begin{equation}\label{eq:cost_LQ_optimal}
c(m_t,\hat m_t,a_t):=(m_t - \hat m_t)^\intercal (m_t - \hat m_t)z(a_t) + \ell(m_t,a_t),
\end{equation}
  where $\hat m_t:=h_t(o_{1:t}, a_{1:t-1}) \in \mathbb{R}^{Dd_s}$ is a generic  nonlinear estimator and $z$, $\ell$ are real-valued non-negative functions, i.e., $z: \mathcal{A} \rightarrow \mathbb{R}_{\geq 0}$ and $\ell: \mathbb{R}^{Dd_s} \times \mathcal{A} \rightarrow \mathbb{R}_{\geq 0}$.
\subsection{Problem statement}

Let $J(g)$ be the total expected discounted cost given by:
\begin{equation}\label{eq:J}
J(g)=\mathbb{E}^{g}[\sum_{t=1}^\infty \gamma^{t-1} c(m_t,\hat m_t, a_t)],
\end{equation}
where  $\gamma \in (0,1)$ is the discount factor (which is  an  incentive parameter to push  for a decision early  rather than postponing it  indefinitely) and  the expectation in~\eqref{eq:J} is taken with respect to  the probability measures induced on  the sample paths by  the choice of strategy $g$. 
\begin{problem}\label{prob:POMDP}
Given  $\varepsilon \in \mathbb{R}_{>0}$,  it is desirable to develop  an $\varepsilon$-optimal strategy $g^\ast_{\varepsilon}$ such that for any  strategy $g$,
\begin{equation}\label{eq:epsilon-optimal}
J(g^\ast_{\varepsilon}) \leq J(g)+\varepsilon.
\end{equation}
\end{problem}

\begin{problem}\label{prob:RL}
When  the  knowledge of the model  $(f,q,c,P_{\mathbf w})$ is incomplete,  it is desirable  to find a reinforcement learning (RL) algorithm that converges to  an  $\varepsilon$-optimal strategy $g^\ast_{\varepsilon}$  satisfying inequality~\eqref{eq:epsilon-optimal}.
\end{problem}

In what follows,  we  first study the Markov-chain model  and then investigate the special case of linear dynamics.
%\textcolor{red}{
%it can be viewed as $\frac{1}{n}\sum_{i=1}^n \hat c(x^i_t,\hat h(\hat x^i_t),a_t)=c(m_t,\hat m_t, a_t)$
%}

% The first extension is where there is one decision maker  and  $k$ different estimators  $h^1,\ldots,h^k$,   $k\in \mathbb{N}$, and the second one is where there are $k$ pairs of decision makers and estimators with additive cost function. For the simplicity of presentation, these extensions are  not considered. 
% 
%\begin{remark}
%\emph{There is no loss of generality in restricting attention to the case of a single decision maker because  when there are  multiple  decision makers,  the optimization problem of each decision maker is decoupled from others due to the fact that the state dynamics does not depend on the decision makers.  More precisely,   each decision maker can independently solve  Problems~\ref{prob:POMDP} and~\ref{prob:RL} with possibly different parameters.}
%\end{remark}

\section{Dynamics of Data}\label{sec:preliminary}
Prior to addressing Problems~\ref{prob:POMDP} and~\ref{prob:RL}, it is necessary to analyze the evolution of data in time.   Many natural systems obey some form of the invariance principle. For example, the  outcome of an election is independent of voters' identity (index of the voters), the  spectrum of the adjacency matrix of an undirected graph does not depend  on  the specific labeling of the nodes,  and   the power demand of  a user in a   smart grid is often independent of   other users' demands. Hence,  it is  reasonable to assume that the local noises are exchangeable  (index-invariant) and i.i.d. (independent  and identically distributed). It is to be noted that  these  are standard assumptions in  statistical models and data science~\cite{chow2012probability}.
\begin{assumption}\label{assum:exchangeable}
The primitive random variables $w^1_t,\ldots,w^n_t$ are exchangeable  at any time $t \in \mathbb{N}$.
\end{assumption}

\begin{proposition}\label{thm:mean_field_evolution_exchangeable}
Let Assumption~\ref{assum:exchangeable} hold. The empirical distribution $m_t$, $t \in \mathbb{N},$ is a Markov process and evolves almost surely at   any  state $s \in \mathcal{S}$ as follows:
\begin{equation}
\Compress
m_{t+1}(s') = \sum_{s \in \mathcal{S}}  \sum_{w \in \mathcal{W}} m_t(s)\ID{f(s,m_t,w)=s'}  [ \frac{1}{n} \sum_{j=1}^n \ID{w^i_t=w}].
\end{equation}
\end{proposition}
\begin{proof}
The proof follows directly  from  the fact  that the empirical distribution of states  and also  the primitive random variables  are invariant to the permutation of nodes.
\end{proof}
\begin{assumption}\label{assum:iid}
The primitive random variables $w^1_t,\ldots,w^n_t $ are i.i.d. at any time $t \in \mathbb{N}$  with probability  function $P_W$.
\end{assumption}

Under Assumption~\ref{assum:iid}, the  dynamic equation~\eqref{eq:model-dynamics} can also  be described in terms of the transition probability matrix at  any $s',s \in \mathcal{S}$ and $m \in \Mspace$ as follows:
\begin{align}\label{eq:transition_prob}
T(s',s,m)&:= \Prob{s^i_{t+1}=s'|s^i_t=s, m_t=m} \\
&= \sum_{w \in \mathcal{W}} P_W(w^i_t=w)\ID{s'=f(s,m,w)},
\end{align}
where the probability of  transitioning to state $s'$ from state $s$, given the empirical distribution $m$,  is equal  to  the probability of  realizations $w$ resulting in this transition.  For any $s',s \in \mathcal{S}$ and $m \in \Mspace$, define the vector-valued function $\phi_{m(s)}: \mathcal{S}^2 \times \Mspace \rightarrow \mathcal{P}\left(\{0,1,\ldots, n \cdot  m(s)\}\right)$ as follows:
\begin{multline}\label{eq:phi}
\phi_{m(s)}(s',s,m)= \delta_0(n \cdot m(s))  \\
 +\ID{m(s) > 0} \Binopdf{\left( n\cdot  m(s), T(s',s,m) \right)},
\end{multline} 
where $\delta_0(n \cdot m(s)) $ is a Dirac measure with the domain set $\{0,1,\ldots,$ $ n \cdot m(s)\}$  and  a unit mass concentrated at zero. In addition, let $\bar \phi: \mathcal{S} \times \Mspace \rightarrow \mathcal{P}\left(\{0,1,\ldots, n\}\right)$ be the convolution of $\phi_{m(s)}(s',s,m)$ over all states $s \in \mathcal{S}=\{s_1,\ldots,s_{|\mathcal{S}|}\}$, i.e.,
\begin{equation}\label{eq:tilde-phi}
\bar \phi(s',m)= \phi_{m(s_1)}(s',s_1,m) \ast \ldots *\phi_{m(s_{|\mathcal{S}|})}(s',s_{|\mathcal{S}|},m),
\end{equation}
where $\bar \phi(s',m)$ is a vector of size $n+1$. When the primitive random variables are  independent  and identically distributed,  the evolution of  data  has a special  structure  as described in the next theorem. 
\begin{theorem}[Deep Chapman-Kolmogorov equation~\cite{Jalal2019MFT}]\label{thm:mean_field_iid}
Let Assumption~\ref{assum:iid} hold.   Given $m_t \in \Mspace$ at time $t \in \mathbb{N}$, the transition probability matrix of the  empirical distribution can be obtained as follows: 
\begin{equation}
\Prob{m_{t+1}(s')=\frac{y}{n} \mid m_t}= \bar \phi(s',m_t)(y+1),  \hspace{.1cm} s' \in \mathcal{S}, y \in \N_{n}.
\end{equation}
\end{theorem}
%\begin{proof}
%From~\eqref{eq:mf-def} and~\eqref{eq:model-dynamics}, it follows that for any $s' \in \mathcal{S},t \in \mathbb{N}$,  
%\begin{equation}\label{eq:proof_tm_iid}
%n \cdot m_{t+1}(s')= \sum_{s \in \mathcal{S}} \sum_{i=1}^n \ID{s^i_t=s} \ID{f(s,m_t,w^i_t)=s'}.
%\end{equation}
%The proof follows from the fact that the probability  mass function of~\eqref{eq:proof_tm_iid}, which is a sum of independent Bernoulli  random variables,   can be described by the convolution  of  the probability mass functions of the random variables.
%\end{proof}
\begin{proof}
The proof is presented in Appendix~\ref{sec:proof_thm:mean_field_iid}.
\end{proof}

To simplify the notation, denote by  $\TM(m',m)$  the  transition probability matrix of the empirical distribution described in Proposition~\ref{thm:mean_field_evolution_exchangeable} and Theorem~\ref{thm:mean_field_iid}, i.e.,  
\begin{equation} \label{eq:tm}
\TM(m',m):=\Prob{m_{t+1}=m' \mid m_t={m}}, \quad m',m \in \Mspace.
\end{equation}

In general, the complexity of computing $\TM$  in time is exponential with respect to the number of nodes $n$.  However, when the local  noises are exchangeable,  this complexity reduces to polynomial time according to Proposition~\ref{thm:mean_field_evolution_exchangeable} (because the space of empirical distributions grows  polynomially with respect to~$n$~\cite{arabneydi2016new}). The above complexity can be  further alleviated in time  when  the noises are i.i.d., according to Theorem~\ref{thm:mean_field_iid}.

\begin{theorem}\label{cor:iid}
 Suppose Assumption~\ref{assum:iid} holds and  the dynamics of the states in \eqref{eq:model-dynamics}  are decoupled, i.e., $s^i_{t+1}=f(s^i_t,w^i_t)$, $i \in \mathbb{N}_n, t \in \mathbb{N}$.  The transition probability matrix $\TM(m',m)$, whose size $|\Mspace|^2$ increases with $n$, can be identified  by the transition probability matrix $T(s',s)$, whose size $ |\mathcal{S}|^2$  is independent of $n$.
\end{theorem}
\begin{proof}
The proof follows directly from~\eqref{eq:phi} and Theorem~\ref{thm:mean_field_iid}, on noting that  $T(s',s,m)$ reduces to $T(s',s)$  for decoupled dynamics. In this case, knowing  function $T(s',s)$, $s',s \in \mathcal{S}$, is enough  to compute the function $\phi_{m(s)}$,  $m \in \Mspace$, and  subsequently function $\bar \phi$ in~\eqref{eq:tilde-phi}. Therefore,   one can determine the global interactions $\TM(m',m)$, $m',m \in \Mspace$  by  identifying the  local interactions $T(s',s)$, $s',s \in \mathcal{S}$.
\end{proof}
\begin{remark}
\emph{A consequence of Theorem~\ref{cor:iid} is that  the  transition probability matrix of data $\TM$ can be identified by  $|\mathcal{S}|^2$ scalars, which is a  considerable reduction in the parameter space.}
\end{remark}

\section{Near-optimal strategies for  Problems~\ref{prob:POMDP} and~\ref{prob:RL}}\label{sec:main}
In this section, we propose near-optimal strategies for Problems~\ref{prob:POMDP} and~\ref{prob:RL}. At any time $t \in \mathbb{N}$, let  $x_t \in \Mspace$  denote the last credible data (i.e., the last observation of the decision maker that is not blank).
 Let  also $y_t \in \N $  be  the number of blanks  up to  time  $t$, after the last credible data  $x_t$. The initial value of $(x_t,y_t)$ is $(m_1,0)$ because $o_1=m_1$. When  data are credible upon request (i.e. $q=1$), the number of blanks $y$   has an inverse  relationship with  the frequency of collecting data that is $1/(y+1)$.
 
  In the following lemma, we identify the dynamics of  $(x_{t+1},y_{t+1})$  at  time $t \in \mathbb{N}$, given  $(x_t,y_t,o_{t+1})$.
  \begin{lemma}\label{lemma:hat f}
  There exists a function $\hat f: \Mspace \times \N  \times \mathcal{O} \rightarrow   \Mspace \times  \N  $ such that
%\begin{equation}
$(x_{t+1}, y_{t+1})=\hat f(x_{t}, y_{t}, o_{t+1}),  t \in \mathbb{N}$, i.e.,
%\end{equation}

\begin{equation}\label{eq:hat-f}
\hat f(x_{t}, y_{t}, o_{t+1}):=\begin{cases}
(x_t, y_t+1), & o_{t+1}=\B, \\
(o_{t+1},0), & o_{t+1} \neq \B.
\end{cases}
\end{equation}
 \end{lemma} 
 \begin{proof}
 The proof follows from the definition of $x_t$ and $y_t$.
 \end{proof}
Now,  let the observation $o_t \in \mathcal{O}$, $t \in \mathbb{N}$,   be rewritten as:
\begin{equation}\label{eq:o-x-y}
o_t= x_t,  \text{if } y_t=0, \quad \text{and}\quad 
o_t=\B, \text{if } y_{t} \neq 0.
\end{equation}
According to~\eqref{eq:hat-f} and~\eqref{eq:o-x-y},  one can conclude that  sets $\{o_{1:t}, a_{1:t-1}\}$  and $\{x_{1:t}, y_{1:t}, a_{1:t-1}\}$  have equivalent information as  each set can be fully specified by the other one, i.e.,
% it results from the Chapman–Kolmogorov equation that
 \begin{equation}\label{eq:chap-kolm}
\Prob{m_t\mid   o_{1:t},  a_{1:t-1}}= \Prob{m_t\mid   x_{1:t}, y_{1:t},  a_{1:t-1}}=\TM^{y_t}(m_t,x_t),
 \end{equation}
 where $\TM^{y_t}$ is  the transition matrix~\eqref{eq:tm} to  the power of~$y_t$. 
 
 In the next lemma, we demonstrate that $(x_{t+1},y_{t+1})$  has  Markovian dynamics by showing that the  posterior probability of $o_{t+1}$ given the history $(x_{1:t},y_{1:t},a_{1:t})$ depends only on the  information at time $t$, i.e.,  $(x_t,y_t,a_t)$.
\begin{lemma} \label{lemma:tansition_prob_o}
Let Assumption~\ref{assum:exchangeable} hold.  Given  any realization $x_{1:t}, y_{1:t}$ and $a_{1:t}$, $t \in \mathbb{N}$,  the following equality holds irrespective of  the strategy~$g$:
\begin{multline}\label{eq:tansition_prob_o}
\Prob{o_{t+1} \mid x_{1:t}, y_{1:t}, a_{1:t}}= (1-a_tq) \cdot \ID{o_{t+1}=\B} \\
+ a_t \cdot  q \cdot    \TM^{y_t+1}(o_{t+1},x_t) \cdot  \ID{o_{t+1} \neq \B}.
%&=\Prob{o_{t+1} \mid x_{t}, y_{t}, a_{t}}.
\end{multline}
\end{lemma}
 \begin{proof}
 The proof is presented in Appendix~\ref{sec:proof_lemma:tansition_prob_o}. 
 \end{proof}
We now prove that the conditional expectation of the per-step cost given the history of information by time $t$ can be described  by  the  information at time $t$.
 \begin{lemma}\label{lemma:hat c}
 Given any realization $x_{1:t}, y_{1:t}$ and $a_{1:t}$, $t \in \mathbb{N}$, there exists a function $\hat c: \Mspace \times  \N \times \mathcal{A} \rightarrow \mathbb{R}_{\geq 0}$ such that:
 \begin{multline}\label{eq:hat-c}
  \Exp{c(m_t,\hat m_t, a_t) \mid x_{1:t}, y_{1:t}, a_{1:t}} =\hat{c}(x_t,y_t,a_t)\\
  :=
 \sum_{m \in \Mspace} c(m,h(\TM^{y_t}(m,x_t)), a_t) \TM^{y_t}(m,x_t),
 \end{multline}
 where the above equality holds irrespective  of strategy $g$.
 \end{lemma}
 \begin{proof}
 The proof is presented in Appendix~\ref{sec:proof_lemma:hat c}.
 \end{proof}
From the results of Lemmas~\ref{lemma:hat f}--\ref{lemma:hat c}, an  optimal strategy  for Problem~\ref{prob:POMDP} is identified by the  following theorem.
\begin{theorem}\label{proposition:bellman}
Let Assumption~\ref{assum:exchangeable} hold.  The  optimal  solution of Problem~\ref{prob:POMDP} is given by the following Bellman equation such that for  any $x  \in \Mspace$ and $y \in \N$,
\begin{equation}\label{eq:bellman_infinite}
V(x,y)=\min_{a \in \mathcal{A}} \left( \hat c(x,y,a) + \gamma \Exp{V(\hat f(x,y,o))} \right),
\end{equation}
where the above expectation is taken over all observations $o \in \mathcal{O}$ with probability distribution~\eqref{eq:tansition_prob_o}.
\end{theorem}
\begin{proof}
The proof follows from the fact that $(x_t,y_t)$ are sufficient statistic to identify the optimal solution of  Problem~\ref{prob:POMDP}. More precisely, given  observations $o_{1:t}$, $t \in \mathbb{N}$, state  $(x_t,y_t)$  is observable at any time $t$ and has Markovian dynamics according to  Lemma~\ref{lemma:hat f}.  In addition,   given  any  history  $(x_{1:t},y_{1:t},a_{1:t})$,  the conditional probability of  $o_{t+1}$  and   the conditional expectation of the per-step cost $c(m_t,\hat m,a_t)$  do not depend on the strategy $g$ according to Lemmas~\ref{lemma:tansition_prob_o} and~\ref{lemma:hat c}, respectively. Consequently,  the optimal solution of Problem~\ref{prob:POMDP} can be identified by dynamic programming decomposition, and  Bellman equation~\eqref{eq:bellman_infinite}  is obtained from well-known results  in Markov decision theory~\cite[Proposition 5.4.1,  Volume 1]{Bertsekas2012book}.
\end{proof}

Since the planning space in Theorem~\ref{proposition:bellman} is countably infinite,  the solution of the Bellman equation~\eqref{eq:bellman_infinite}  is intractable,  in general.  As a result,  we  propose an $\varepsilon$-optimal solution  based on a truncation technique whose performance is within an arbitrary neighbourhood (determined by $\varepsilon$)  of the optimal performance for Problem~\ref{prob:POMDP}.  

\begin{remark}
\emph{Under some practical constraints  such as limited energy resources  for data collection   or saturation of states,   the feasible set of the dynamic program~\eqref{eq:bellman_infinite} may reduce to  a finite set, yielding  a tractable optimization problem.   
}
\end{remark}

\subsection{An $\varepsilon$-optimal solution for Problem~\ref{prob:POMDP}}
%Since the planning space in Theorem~\ref{proposition:bellman} is countably infinite,  the solution of the Bellman equation~\eqref{eq:bellman_infinite}  is intractable,  in general.  As a result,  we  propose an $\varepsilon$-optimal solution  based on a truncation technique~\cite{Senott:1999} whose performance is within an arbitrary neighbourhood (determined by $\varepsilon$)  of the optimal performance associated with the solution of Problem~\ref{prob:POMDP}.  

 Denote by $c_{\text{max}}$ an upper bound on the per-step cost~\eqref{eq:per_step_cost_general}.  For any $k \in \mathbb{N}$, define the following Bellman equation at  any  $x,m^\ast \in \Mspace$ and $y \in  \N_k$:
 \begin{equation}\label{eq:Bellman_approximate}
\tilde V_k(x,y)=\min(\tilde V^0_k(x,y),\tilde V^1_k(x,y)),
\end{equation}
where 
 \begin{equation}
 \begin{cases}
\tilde V^0_k(x,y):=\sum_{m \in \Mspace}   c(m,h(\TM^{y}(m,x)),0)  \TM^{y}(m,x)\\
\qquad + \gamma (\ID{y <k} \tilde V_k(x,y+1) 
+ \ID{y=k} \tilde V_k(m^\ast,0)),\\
\tilde V^1_k(x,y):=\sum_{m \in \Mspace}   c(m,h(\TM^{y}(m,x)),1)  \TM^{y}(m,x)\\
+  (1-q) \gamma (\ID{y <k} \tilde V_k(x,y+1) + \ID{y=k} \tilde V_k(m^\ast,0))\\
\qquad + q \gamma (\sum_{m' \in \Mspace} \TM^{y+1}(m',x) \tilde V_k(m',0)).
\end{cases}
\end{equation}

 \begin{theorem}\label{thm:approximate-pomdp}
Suppose Assumption~\ref{assum:exchangeable} holds. Given any $\varepsilon \in \mathbb{R}_{>0}$,  let   $k(\varepsilon) \in \mathbb{N}$   be sufficiently large such that
\begin{equation}\label{eq:k-large}
 k(\varepsilon) \geq \log (\frac{(1-\gamma)\varepsilon}{2c_{\text{max}}}) / \log(\gamma).
\end{equation} 
Then,  using  $k=k(\varepsilon)$, any  solution  to  the Bellman equation~\eqref{eq:Bellman_approximate} is an $\varepsilon$-optimal  solution for Problem~\ref{prob:POMDP}, i.e.,
\begin{equation}
g^\ast_\varepsilon(x,y):=\begin{cases}
0, & \tilde V^0_k(x,y) \leq \tilde V^1_k(x,y),\\
1, &  \tilde V^0_k(x,y) > \tilde V^1_k(x,y).
\end{cases}
\end{equation}
 \end{theorem}
 \begin{proof}
 The proof is presented in Appendix~\ref{sec:proof_thm:approximate-pomdp}. 
 \end{proof}
To  numerically compute the solution of the Bellman equation~\eqref{eq:Bellman_approximate}, one can use  value iteration,  policy iteration,  or any other  existing approximate  method~\cite{Bertsekas2012book,toral2014stochastic}.  Note that the space of  Bellman equation~\eqref{eq:Bellman_approximate}  (i.e., $\Mspace \times \N_k$) grows polynomially  with  the number of nodes $n$  and linearly with  the approximation index $k$.  For the case when~\eqref{eq:k-large} is equality,   $k(\varepsilon)$ is proportional to $-\log (\text{constant}\times  \varepsilon)$. Note that  $\gamma \in (0,1)$, which means its log is negative. Hence, $-\log (\text{constant} \times \varepsilon)=\log (\text{constant} \times {\varepsilon}^{-1})$. Therefore,  the following result holds.
 \begin{corollary}
 The computational complexity of the proposed solution in Theorem~\ref{thm:approximate-pomdp} is linear with respect to the approximation index $k(\varepsilon)$, logarithmic with respect to the inverse of the size of the desired neighborhood $\varepsilon$, and polynomial with respect to the number of nodes $n$.
 \end{corollary}

\begin{corollary}\label{cor:n=1}
For the special case of single node, i.e., $n=1$,   there is no loss of optimality in replacing space $\Mspace$ by space $\mathcal{S}$ in Theorems~\ref{proposition:bellman} and~\ref{thm:approximate-pomdp}.
\end{corollary}
 \begin{proof}
The proof follows on noting that when $n=1$,  spaces $\Mspace$ and $\mathcal{S}$ have equivalent information, i.e.,  $m_t=\delta_{s_t}$ at any time $t \in \mathbb{N}$.
 \end{proof}

\begin{remark} 
\emph{
It is to be noted that the application domain of the present work is different from  applications such as sensor selection, where the objective is to dynamically choose a subset of sensors in order to monitor a time-varying phenomenon~\cite{wu2008optimal,fuemmeler2011sleep}. For example, in sleep sensor scheduling control~\cite{fuemmeler2011sleep} the dynamics of the target (phenomenon) is decoupled from the scheduling (sampling) strategy, the value of  data at each time instant is binary (i.e., it is  zero if data is observed and  it is one otherwise), and the planning space consists of the belief state of the target  as well as the residual sleep times of sensors~\cite[Theorem 3.1]{fuemmeler2011sleep}.    In contrast, the phenomenon considered in this paper is  an  estimate of data, generated by the estimator,  that is influenced by the sampling strategy, and  leads to  a dual effect~\cite{bar1974dual}. The value of data  depends on various parameters such as the cost and estimator functions and is not necessarily a binary variable.  Furthermore,  the proposed dynamic program is based on a  planning space, defined as the last credible data and the elapsed time since then (which is not a belief space).  In addition,  the action set here does not depend on the number of nodes  and the state space and  the dynamics of the phenomenon (data), and the solution methodology  is  amenable to the incompleteness of the model structure.}
\end{remark}
%More importantly, the focus of this paper is to provide an $\varepsilon$-optimal strategy  so that the solution methodology  is  amenable to the incompleteness of the model structure, to be discussed in the next subsection.

\subsection{An $\varepsilon$-optimal solution for Problem~\ref{prob:RL}}
Two different approaches are considered here  to  find a near-optimal  strategy for Problem~\ref{prob:RL}. The first one is an indirect  (model-based)  approach which involves two steps: supervised learning and planning.  Given a large number of training samples, one can utilize supervised-learning (parametrization)  techniques such as  linear  regression and  logistic regression   to  identify the model, and  then  find the planning solution of the Bellman equation~\eqref{eq:Bellman_approximate}  by  using  methods such as value iteration and policy iteration~\cite{Bertsekas2012book}. In general, the total  number of unknown  parameters that should be learned  to solve  equation~\eqref{eq:Bellman_approximate}  using this  approach is equal to $ |\Mspace|^2+ 2(k+1) |\Mspace|+ 1$, with $|\Mspace|^2$ scalars  for transition probability matrix $\TM$,  $|\Mspace| |\N_k||\mathcal{A}|$ scalars for the per-step cost $c$, and  $1$ scalar for the credibility of data~$q$.    In practice, the first approach  is feasible  when the number of unknown parameters is relatively small.
 For example,  the deep Chapman-Kolmogorov equation in~Theorem~\ref{thm:mean_field_iid} can be parametrized by  a small number of variables  according to~Theorem~\ref{cor:iid},  for the case when the dynamics of nodes are decoupled  and the random variables are  i.i.d., which is a non-trivial (yet efficient)  parametrization. 
%
%The first one is an indirect approach where machine learning methods are primarily used   to learn the model, and then  the solution of Bellman equation~\eqref{eq:Bellman_approximate}  is obtained by various methods such as value iteration and policy iteration.  

% In the same spirit,~\cite{} uses Dirichlet distribution to  learn the model  and then use point-based methods to find a good solution in the belief space.
\begin{proposition}\label{thm:machine}
Let Assumption~\ref{assum:iid} hold and  the per-step cost function as well as the estimator function be given.     When the dynamics of the states are decoupled and $n$ is large,  it is more efficient to first learn the model and then solve the Bellman equation~\eqref{eq:Bellman_approximate}  to obtain an $\varepsilon$-optimal solution for Problem~\ref{prob:RL}.
\end{proposition}
\begin{proof}
The proof follows on noting that the total  number of unknown  parameters  in~\eqref{eq:Bellman_approximate}  is $|\mathcal{S}|^2+1$, which   is  independent of the number of nodes $n$ (i.e., $|\mathcal{S}|^2$ parameters  correspond to the transition probability matrix $\TM$  from  Theorem~\ref{cor:iid} and  one parameter corresponds to the credibility probability~$q$).
\end{proof}

The second approach is a  direct (model-free) method that finds an $\varepsilon$-optimal solution of the Bellman equation~\eqref{eq:Bellman_approximate} without learning the model. In this approach,  any approximate dynamic programming method such as TD($\lambda$) and Q-learning can be employed to find a sufficiently close approximation to the  Bellman equation~\eqref{eq:Bellman_approximate}.  To this end, it is important to note that   the state space of Bellman equation \eqref{eq:Bellman_approximate}  is finite and its states are observable by virtue of a model-independent function $\tilde f$ given by~\eqref{eq:tilde-f}.   To illustrate the approach,  we use the Q-learning algorithm in this paper as a model-free reinforcement learning method. We present Algorithm~\ref{alg.RL-Qlearning} which requires only $|\Mspace| |\N_k||\mathcal{A}|$ scalars for the Q-functions, where $|\N_k| |\mathcal{A}|$ is, in fact, equal to $2(k+1)$.

\alglanguage{pseudocode}
\begin{algorithm}[t!]
\small
\caption{Proposed Q-Learning Procedure}
\label{alg.RL-Qlearning}
\begin{algorithmic}[1]
\State  Given any $\varepsilon \in \mathbb{R}_{>0}$,  let   $k(\varepsilon) \in \mathbb{N}$  satisfy inequality~\eqref{eq:k-large}.
%\begin{equation}
% k:=k(\varepsilon) \geq \log \left(\frac{(1-\gamma)\varepsilon}{2c_{\text{max}}}\right) / \log(\gamma).
%\end{equation} 
\State Let $x_1=m_1$, $y_1=0$, $Q_1(x,y,a)=0$ and $\alpha_1(x,y,a)=1, \forall x \in \Mspace, y \in \N_k, a \in \mathcal{A}$.

\State At  iteration $\tau \in \mathbb{N},$ given any state $(x,y) \in \Mspace \times \N_k$ and any action $a \in \mathcal{A}$,  update   the  corresponding Q-function and  learning rate as follows:
\begin{align}
\begin{cases}
Q_{\tau+1}(x,y, a)=(1-\alpha_\tau(x,y, a)) Q_\tau (x,y, a)\\
\quad + \alpha_\tau (x,y, a) (c'+\gamma \min_{a' \in \mathcal{A}} Q_\tau(x',y',a')),\\
\alpha_{\tau+1}(x,y,a)=\lambda({\tau},\alpha_\tau(x,y,a)),
\end{cases}
\end{align}
where $c'$ is the immediate cost, $(x',y')=\tilde f(x,y,o)$ is the next state with  observation $o$, and $\lambda$  determines proper learning rates  $\alpha_\tau \in [0,1], \tau \in \mathbb{N},$  such that
%\begin{equation}
$\sum_{\tau=1}^\infty \alpha_\tau(x,y,a) = \infty$ and  $\sum_{\tau=1}^\infty (\alpha_\tau(x,y,a))^2 <\infty.$
%\end{equation}

  \State \noindent Let $\tau= \tau+1$, and go to step 3 until the termination condition is satisfied.
\Statex
\end{algorithmic}
  \vspace{-0.2cm}%
\end{algorithm}

\begin{theorem}\label{thm:RL}
 Let Assumption~\ref{assum:exchangeable} hold, and  suppose that every pair of state $(x,y) \in \Mspace \times \N_k$ and action $a \in \mathcal{A}$ is visited infinitely often  in Algorithm \ref{alg.RL-Qlearning}. Then,  the following results hold:
\begin{itemize}
\item[ (a)]  For any $(x,y,a) \in \Mspace \times \N_k \times \mathcal{A}$, the Q-function $Q(x,y,a)$  converges  to  $Q^\ast(x,y,a)$ with probability one. 
\item[ (b)] Let $g^\ast_\varepsilon \in \argmin_{a \in \mathcal{A}}Q^\ast(x,y,a)$ be a greedy strategy; then,   $g^\ast_\varepsilon$ is an $\varepsilon$-optimal strategy  for Problem~\ref{prob:RL}.
\end{itemize}
\end{theorem}
\begin{proof}
The proof is presented in Appendix~\ref{sec:proof_thm:RL}.
\end{proof}

\section{Networks with a Large Number of Nodes: Asymptotic Analysis} \label{sec:asymptotic}
In this section, we show  that estimating data tends to  be more efficient than collecting  it when the number of nodes  is sufficiently large. To this end, the following  assumption is made on the model.
\begin{assumption}\label{assum:Lischitz}
There exist positive scalars $K_T$ and $K_c$ such that for any $s', s \in \mathcal{S}$, $a \in \mathcal{A}$, and $m, m' \in \mathcal{P}(\mathcal{S})$,
\begin{itemize}
\item[1.] $|T(s',s,m) - T(s',s,m') | \leq K_T\Ninf{m - m'}, $
\item[2.] $c(m, \hat m,a) \leq K_c \Ninf{m - \hat m}.$
\end{itemize}
\end{assumption}
It is to be noted that Assumption~\ref{assum:Lischitz} is mild because  any polynomial function in $m$ is  Lipschitz   with respect to  $m$ due to the fact  that $m$ is confined to a bounded interval.    Let  $\bar T: \mathcal{P}(\mathcal{S}) \rightarrow \mathcal{P}(\mathcal{S})$ be defined as follows:
\begin{equation}\label{eq:h_large_T}
\bar T(p):= \sum_{s \in \mathcal{S}}  p(s) T(\boldsymbol \cdot,s,p), \quad p \in \mathcal{P}(\mathcal{S}).
\end{equation}
 Let  also the infinite-population estimator $h$ be defined as:
\begin{equation}\label{eq:h_large_n}
%\hat m_{t}= \underbrace{\bar T(...\bar T}_{t-1}(\hat m_1)), \quad t \in  \mathbb{N}\backslash \{1\}.
\hat m_{t}= \underbrace{\bar T \circ  \ldots \circ \bar T}_{t-1}(\hat m_1), \quad t \in  \mathbb{N}\backslash \{1\},
\end{equation}
where $\circ$ is  the composition  operator.
%Notice that  the estimator~\eqref{eq:h_large_n}   has an open-loop structure because it requires no observation.
\begin{remark}\label{remark:linear_inf}
\emph{When the dynamics are decoupled and $n=\infty$, the nonlinear dynamics~\eqref{eq:h_large_T} reduces to a linear equation, i.e., $\bar T(p)= T(\boldsymbol \cdot, \boldsymbol \cdot) p$.}
\end{remark}
\begin{lemma}\label{lemma:estimator_lipschitz}
Let Assumption~\ref{assum:Lischitz} hold.  There exists a positive scalar $K_p$ such that
\begin{equation}
\Ninf{ \bar T( p)- \bar T( \hat p)} \leq K_p \Ninf{p-\hat p}, \quad p,\hat p \in \mathcal{P}(\mathcal{S}).
\end{equation}
\begin{proof}
The proof follows from equation~\eqref{eq:h_large_n}, Assumption~\ref{assum:Lischitz}, and the fact that the Lipschitz property is preserved  under summation and multiplication.
\end{proof}
\end{lemma}
%\edit{In the next lemma, it is shown that the conditional expectation of the error between the real data $m_{t+1}$ and its estimate $\hat m_{t+1}$  is bounded, given  the current error at time $t$.}
\begin{lemma}\label{lemma:evolution_mean_field}
Let Assumption~\ref{assum:iid} hold. Given any $m_t \in \Mspace$ and  $\hat m_t \in \mathcal{P}(\mathcal{S})$,  $t \in \mathbb{N}$,  the following inequality is satisfied:
\begin{equation}
\Exp{\Ninf{m_{t+1} - \hat m_{t+1}} |m_t, \hat{m}_t} \leq K_p \Ninf{m_t -\hat m_t} + \mathcal{O}(\frac{1}{\sqrt n}).
\end{equation}
\end{lemma}
\begin{proof}
The proof follows  directly from the triangle inequality, Lemma~\ref{lemma:estimator_lipschitz} and  \cite[Lemma~2]{JalalCDC2017}. 
%\begin{align*}
%&\Exp{\Ninf{m_{t+1} - \hat m_{t+1}} |m_t, \hat{m}_t} \\
%&\leq  \Exp{\Ninf{m_{t+1}  - h(m_t)} |m_t, \hat m_t} \hspace{-.05cm}+ \hspace{-.05cm} \Exp{\Ninf{h(m_t)  - \hat m_{t+1}} |m_t, \hat{m}_t} \nonumber \\
%&\substack{(a)\\ =} \mathbb{E} \big|     \|  \sum_{s \in \mathcal{S}} \sum_{w \in \mathcal{W}} m_{t}(s) \ID{f(s,m_t,w)=s'}  \big [P_W(w) \\
%&\quad - \frac{1}{n } \sum_{j=1}^n \ID{w^j_t=w} \big]  \big| 
%+\Exp{\Ninf{h(m_t)  - h(\hat m_{t})}|m_t, \hat{m}_t} \nonumber \\
%& \substack {(b) \\ \leq } \mathcal{O}(\frac{1}{\sqrt n}) + K_h \Ninf{m_t -\hat m_t},
%\end{align*}
%where $(a)$ follows from Theorem~\ref{thm:mean_field_evolution_exchangeable} and equation \eqref{eq:h_large_n} and $(b)$ results from  Lemma~\ref{lemma:estimator_lipschitz}, Assumption~\ref{assum:iid}, and the fact that the empirical distribution converges to its limit in the mean square sense at the rate $\mathcal{O}(\frac{1}{n})$, \cite[Lemma~2]{JalalCDC2017}. 
\end{proof}

%The following assumption is made  to ensure that the cost function is bounded.
\begin{assumption}\label{assump:gamma_k}
Assume that $\gamma K_p  <1$. Note that this inequality always holds when the dynamics of states in~\eqref{eq:model-dynamics} are decoupled  because in this case $K_p=1$ satisfies Lemma~\ref{lemma:estimator_lipschitz}.
\end{assumption}

\begin{theorem}\label{thm:asymptotic}
Let   Assumptions~\ref{assum:iid},~\ref{assum:Lischitz} and~\ref{assump:gamma_k} hold. The total expected discounted  cost associated with using the estimator~\eqref{eq:h_large_n} is bounded at all times and converges to zero at the rate $1 /\sqrt n$ as follows:
\begin{equation}\label{eq:estimation_threshold}
\Exp{\sum_{t=1}^\infty \gamma^{t-1} c(m_t,\hat m_t,0)} \leq \frac{\gamma K_c}{(1-\gamma)(1-\gamma K_p)}\mathcal{O}(\frac{1}{\sqrt n}),
\end{equation}
where $\mathcal{O}(\frac{1}{\sqrt n})$  depends on the variance of local noises. 
\end{theorem}
 \begin{proof}
 The proof is presented in~Appendix~\ref{sec:proof_thm:asymptotic}.
 \end{proof}

\begin{definition}[\textbf{Certainty Threshold}]
\emph{The right-hand side of~\eqref{eq:estimation_threshold}  is defined  as the} certainty threshold\emph{. This threshold  depends on the number of nodes  as well as the  Lipschitz constants of the  transition probability matrix and cost function introduced in Assumption~\ref{assum:Lischitz}.}
\end{definition}
%\edito{When  this threshold is  below than  the cost of data collection,   data estimation outweighs the  data collection indefinitely.}
\begin{corollary}\label{cor:certainty}
Let Assumptions~\ref{assum:iid}--\ref{assump:gamma_k} hold.  If the collection cost  is greater than the certainty threshold,
%\begin{equation}\label{eq:condition_thm}
%\frac{\gamma K_d}{(1-\gamma)(1-\gamma K_p)}\mathcal{O}(\frac{1}{\sqrt n})  \leq  \min_{m} \ell(m), 
%\end{equation}
 then the optimal solution to Problem~\ref{prob:POMDP} is to use estimator~\eqref{eq:h_large_n} at all times.
\end{corollary}
\begin{proof}
Consider two scenarios,  where  in the first  one  the strategy is to always use the estimator function~\eqref{eq:h_large_n},  and in the second one the strategy  is to collect data at least once. The costs of  both scenarios  until the first data collection are the same.  The proof  now follows directly  from Theorem~\ref{thm:asymptotic}.
\end{proof}
\begin{remark}[De Finetti's theorem]
\emph{It is worth mentioning that when $n=\infty$,  exchangeable random variables behave as i.i.d. variables,  according to de Finetti's theorem~\cite{diaconis1980finetti}. }
\end{remark}

\section{Linear dynamics: A special case}\label{sec:linearcase}
%
%Since linear models are of particular interest  in  many  real-world applications, we  consider Problems~\ref{} and~\ref{} for the model presented in Subsection~\ref{}.

\subsection{Optimal estimator}
The model presented in Subsection~\ref{sec:linear_model}  has  various applications, e.g., in  remote-state estimation wherein an encoder sends a Markovian process to a decoder over an unreliable link under the Transmission Control Protocol (TCP),  where an acknowledgement  is sent to  the encoder upon receiving the data at the decoder~\cite{lipsa2011remote,nayyar2013optimal}.  The objective of the  encoder and decoder is  to collaborate in such a way that the cost function~\eqref{eq:cost_LQ_optimal}  is minimized, where $g$ is the transmission  strategy of the encoder and $h$ is the estimation strategy of the decoder.    By using  majorization theory  and  imposing some conditions such as  symmetric and unimodal  probability distribution of random variables,   it is shown in~\cite{lipsa2011remote,nayyar2013optimal} that the optimal estimator is Kalman-like. In what follows, we extend the above  findings by a simple proof technique using the proposed planning space, and subsequently establish a separation theorem  that holds for the general case of multi-dimensional  dynamic systems with an arbitrary probability distribution, without resorting to majorization theory or  any other  conditions on the random variables. To this end, define $x_t \in \mathbb{R}^{d_s}$ as the last (credible) observed data and $y_t \in \N$ as the elapsed time since then.  
\begin{theorem}[Separation principle]\label{thm:separation}
The problem of finding the optimal estimator in~\eqref{eq:cost_LQ_optimal} is separated from that of the optimal scheduling strategy. In particular, the optimal  estimator has a structure similar to  the minimum mean-square estimator  with  the following   Kamlan-like  update rule:
%\begin{equation}\label{eq:update_rule_KF-2}
%\hat m_{t+1}=\begin{cases}
%x_{t+1}=o_{t+1}, &  o_{t+1} \neq \B \text{ (i.e., $y_{t+1}=0$)},\\
%A^{1+y_{t}} x_{t}=A \hat m_t, & o_{t+1} = \B \text{ (i.e., $y_{t+1}> 0$)}.
%\end{cases}
%\end{equation}
\begin{equation}\label{eq:update_rule_KF}
\hat m_{t+1}= A \hat m_t+ L(y_{t+1}) (x_{t+1} - A \hat m_t),
\end{equation}
where $L(y_{t+1})=1$ if $y_{t+1}=0$ and $L(y_{t+1})=0$ if $y_{t+1} \neq 0$. This result holds regardless of the sampling strategy,   the order of dynamics, the probability distribution of the initial states, credibility of data and  additive noises. 
\end{theorem} 
\begin{proof}
The proof is presented in Appendix~\ref{sec:proof_thm:separation}.
\end{proof}
\begin{remark}
\emph{The result of Theorem~\ref{thm:separation} also holds for the time-varying finite-horizon case since the proof technique presented above does not  depend on the time-homogeneity of the model.}
\end{remark}

\subsection{Sampling (scheduling) strategy}
Consider a centralized networked control system wherein  the joint state is perfectly observed  and the optimal joint action is a state-feedback strategy. In  such a case,   every  node must  broadcast  its state  at each  time instant   so that all nodes   can observe the joint state in order to  compute  their control actions accordingly. In practice, however, sharing information is  costly, the quality of data transmission is sometimes compromised,   and  data packets may be  lost in  the communication channels (e.g., in  erasure  channels). Thus,  it is important to be able to  implement   the centralized  solution in a  distributed manner. To this end,   each node can
%\cite{xu2004optimal} 
   solve  a (local) scheduling problem  in order to decide when to  broadcast its state~\cite{hespanha2007survey}. Once the information is broadcast,  other  nodes  can  update  their  estimates of  the state of  the node,  to be used  in their strategies.  For the case where  every node uses a  minimum mean-square estimator and attention is restricted to a threshold-type  scheduling,    a dynamic-program-based solution  can be developed  whose information state is  the estimation error (that is a continuous variable in $\mathbb{R}^{d_s}$ with a model-dependent dynamics)~\cite{hespanha2007survey}.  It is shown in~\cite{lipsa2011remote,nayyar2013optimal} that such  threshold-type policies are optimal under certain conditions.

 In this work, we use a dynamic program with an information state different from~\cite{hespanha2007survey,lipsa2011remote,nayyar2013optimal}, which is a discrete variable  between $0$ and $k$,  independent of the state space dimension  $d_s$,  with  a model-free dynamics. Therefore, it is computationally easy to find  a near-optimal  strategy by solving our proposed dynamic program.

\begin{assumption}\label{ass:sampling}
Suppose that  the per-step cost~\eqref{eq:cost_LQ_optimal} satisfies the inequality $c_t(m_t-\hat m_t, a_t) \leq c_{\textit{max}}$,  where $c_{\textit{max}}$ is a known positive constant and the estimator is the minimum mean-square estimator, i.e., $\hat m_t=\Exp{m_t|o_{1:t},a_{1:t-1}}$, $t \in \mathbb{N}$. 
\end{assumption}

 From  Assumption~\ref{ass:sampling} and the dynamics of  the mean-square estimator given by~\eqref{eq:update_rule_KF},    it follows that  for any $t \in \mathbb{N}$:
\begin{equation}\label{eq:per-step-8}
m_t - \hat m_t=  \ID{y_t > 0}\sum_{\tau=1}^{y_t} A^{\tau -1} \bar w_{t-\tau}.
\end{equation}
Since $w_{1:\infty}$ is an i.i.d. random process, one has:
\begin{equation}\label{eq:cost:general:LQ}
c_t(y_t, a_t)=\Exp{c_t(m_t - \hat m_t, a_t) \mid x_{1:t}, y_{1:t},a_{1:t}},  
\end{equation}
where  the  above per-step cost does not depend on $x_t$; hence, it can be denoted by $c(y_t,a_t)$, $t \in \mathbb{N}$. It is also possible to  consider a special case of Assumption~\ref{ass:sampling}, described below, that provides an explicit  expression for~\eqref{eq:cost:general:LQ}.  
\begin{assumption}\label{ass:limited-LQ}
For any $t \in \mathbb{N}$, $a_t \in \mathcal{A}$ and $m_t \in \Mspace$, let functions~$z(a_t)$ and~$\ell(m_t,a_t)$ in~\eqref{eq:cost_LQ_optimal} be equal to  $1$ and $\ell a_t $, respectively,  where $\ell \in \mathbb{R}_{\geq 0}$.  In addition,   matrix $A$ is symmetric and all of its eigenvalues are within the unit circle.\footnote{For finite-horizon analysis,  $A$ can be any arbitrary matrix.}
\end{assumption}
Under Assumption~\ref{ass:limited-LQ} and Theorem~\ref{thm:separation},~\eqref{eq:cost:general:LQ} can be  computed as follows:
\begin{multline}\label{eq:per-step-3}
c(y_t,a_t)=\Exp{(m_t - \hat m_t)^\intercal (m_t - \hat m_t) + \ell a_t \mid x_{1:t}, y_{1:t},a_{1:t}}\\
= \ID{y_t > 0}\TR(\sum_{\tau=1}^{y_t} (A^\intercal A)^{\tau-1} \bar{\Sigma}^w ) +\ell a_t,\\
=\TR((\mathbf{I}-A^\intercal A)^{-1} (\mathbf I - (A^\intercal A)^{y_t}) \bar{\Sigma}^w) +\ell a_t.
\end{multline}
where  $\mathbf{I}$ is the identity matrix and $\bar \Sigma^w$  is the covariance matrix of $\bar w_t$, $t \in \mathbb{N}$. The following theorem  is a consequence of Theorems~\ref{thm:approximate-pomdp} and~\ref{thm:RL}.
\begin{theorem}[Sampling strategy]\label{cor:linear}
 Let  either Assumption~\ref{ass:sampling} or  Assumption~\ref{ass:limited-LQ} hold.  There is no loss of optimality in restricting attention  to the space of  elapsed times  after the last credible data (i.e.,  there is no need to know the data). More precisely, select  a sufficiently large  $k \in \mathbb{N}$ such that inequality~\eqref{eq:k-large}  is satisfied,  and  simplify the  dynamic program~\eqref{eq:Bellman_approximate}    as follows:
\begin{equation}\label{eq:dp:example3}
\tilde V_k(y)=\min(\tilde V^0_k(y),\tilde V^1_k(y)), \quad y \in \N_k, 
\end{equation}
where  $\tilde V^0_k(y):=c(y,0)
+ \gamma (\ID{y <k} \tilde V_k(y+1) 
+ \ID{y=k} \tilde V_k(0))$ and  $\tilde V^1_k(y):=  c(y,1)+  (1-q) \gamma (\ID{y <k} \tilde V_k(y+1)$ $
 + \ID{y=k} \tilde V_k(0))
+ q \gamma  \tilde V_k(0)$.  Then, a near-optimal sampling strategy is to collect  data when $\tilde V^0_k(y_t) <  \tilde V^1_k(y_t)$.  A similar relationship holds for the case that model structure is not known completely, where   Q-learning algorithm proposed in Algorithm~\ref{alg.RL-Qlearning} converges to a near-optimal solution.
\end{theorem}
\begin{proof}
The proof follows form the fact that the  per-step cost in~\eqref{eq:cost:general:LQ} and \eqref{eq:per-step-3} does not depend on $x_t$; hence, $x_t$ is irrelevant information. 
%It is to be noted that the per-step cost~\eqref{eq:per-step-3} is uniformly bounded for every $y_t \in \N_k$ due to the  Cauchy–Schwarz inequality  and  Assumption~\ref{ass:limited-LQ}.
\end{proof}

\begin{corollary}
Let Assumptions~\ref{assum:iid},~\ref{assum:Lischitz} and~\ref{ass:limited-LQ} hold. The certainty threshold  in the case of linear dynamics  with  quadratic cost function converges to zero at the rate $1/n$, which is faster than the generic rate of $1/\sqrt n$. In this case, $\bar \Sigma^w=\frac{1}{n} \DIAG(\Sigma^w,\ldots,\Sigma^w)$, where  $\Sigma^w$ is the covariance  matrix of an individual local noise.
\end{corollary} 
\begin{remark}[Time delay]
\emph{All the results presented in  this paper, including Theorems~\ref{thm:separation} and~\ref{cor:linear}, extend naturally to the case where  observations are received with a fixed time delay $\tau \in \N$,  by simply replacing  $y_t$ with  $y_t+\tau$.}
\end{remark}
\begin{remark}[Mean-field approximation]
\emph{When $n=\infty$ and the dynamics are decoupled,   the infinite-population (linear) model presented in Remark~\ref{remark:linear_inf}  may  be used  in Theorem~\ref{cor:linear} to provide a scale-free approximation.}
\end{remark}

\begin{remark}[Age of information]
\emph{Note that the cost function~\eqref{eq:per-step-3} is exponential  with respect to $y_t$,  reflecting the fact that  the quality of the minimum mean-square estimator deteriorates exponentially in the absence of   credible data.  Nonetheless,  it is  possible to consider a simpler  cost function, e.g.,  an  affine or quadratic cost function in $y_t$, for which the minimum of  the right-hand side of the dynamic program~\eqref{eq:dp:example3} can be obtained more efficiently.  This case is then related to  real-time status updating, where  $y_t$ is the \emph{age of information},  representing the freshness of  data, and  the objective is to monitor a phenomenon  of interest in a timely manner~\cite{alberts1997information,kaul2012real,kadota2018scheduling}.  Hence,  the dynamic program~\eqref{eq:dp:example3} and its reinforcement learning version  can be used to  determine a low-complexity near-optimal solution for minimizing the age of information.
}
%This case  is  then related  to  a  recently developed line of research wherein the  performance index  is defined as the age of information,  that is the elapsed time after the last successfully received data~\cite{kadota2018scheduling}. 
\end{remark}

%\edit{
%\begin{remark}[Decoupled dynamics]
%\emph{It is to be noted that when  nodes are decoupled in the dynamics,   the dynamics of data in the Markov chain setup has a linear dynamics in terms of the probability distributions. To demonstrate this, let $\pi_{t}:=\Prob{m_{t}}$, then:
%\begin{equation}\label{eq:DD1}
%\pi_{t+1}=\TM \pi_{t},
%\end{equation}
%and  for the special case of single node:
%\begin{equation}\label{eq:DD2}
%\pi_{t+1}=T \pi_{t},
%\end{equation}
%where $\pi_{t}:=\Prob{s_t}$. Therefore,  if data $m_t$ is replaced by the categorical distribution $\pi_t$ in the objective function~\eqref{eq:cost_LQ_optimal}, then the results of this section holds for the decoupled Markov-chain setup, where the objective is to find the scheduling and estimating strategies  such that  the estimated categorical distribution $\hat \pi_t:=h(o_{1:t},a_{1:t-1})$ provides a reasonable accuracy. In such a case,   matrix $ \bar A$ in our formulation is replaced by the transition probability matrix, e.g., $\TM$ and $T$.  To have a non-trivial case, one can consider an additive  noise  in  equations~\eqref{eq:DD1} and~\eqref{eq:DD2}, representing  the uncertainty in practice.
% }
%\end{remark}
%}

\subsection{Noisy observations  with   Gaussian random variables}
In this subsection, we  show that the presence of measurement noise adversely impacts the tractability  gained by the proposed planning space, and   consequently finding an $\varepsilon$-optimal solution becomes NP-hard. However, we demonstrate that the resultant optimization problem is  a deterministic nonlinear dynamic optimization problem that may  be solved numerically by various  computational tools. Suppose that Assumption~\ref{ass:limited-LQ} holds, and  that local noises are Gaussian. Let $o^i_{t} \in \mathbb{R}^{d_o}$, $d_o \in \mathbb{N}$, be  the noisy observation of node $i$ at time $t$, i.e.,
%\begin{equation}
 $o^i_{t}=Cs^i_t+  \xi^i_t$,
%\end{equation}
where $\xi^i_{1:\infty}$ is an i.i.d. Gaussian random process with zero mean and finite covariance matrix $\Sigma^{i,\xi} \in \mathbb{R}^{d_o \times d_o}$. In addition, it is assumed that the measurement noises  $\{\xi^i_{1:\infty}\}_{i \in \mathbb{N}_n}$ are mutually independent across nodes, and are also independent from the previously defined primitive random variables.  Then,
\begin{equation}
\bar o^d_{t}:= \frac{1}{n} \sum_{i=1}^n v^{i,d} o^i_t= C_d m^d_t + \bar \xi_t^d,
\end{equation}
where $C_d:= \frac{1}{n}\sum_{i=1}^n v^{i,d} C$ and $\bar \xi^d_t:=\frac{1}{n}\sum_{i=1}^n v^{i,d} \xi^i_t$ with the covariance matrix $\bar \Sigma^{\xi,d}:= \frac{1}{n^2}\sum_{i=1}^n (v^{i,d})^2 \Sigma^{i,\xi} $.   For simplicity of  presentation, assume that $q=1$, and that the horizon is finite. Therefore, 
\begin{equation}
 o_{t+1}:= C(a_t) m_{t+1} + E(a_t) \bar \xi_{t+1}, 
\end{equation}
where  $o_t=\VEC(o_t^1,\ldots,o_t^D)$, $C(a_t):=a_t \DIAG(C_1,\ldots,C_D)$, $\bar \xi_t=\VEC(\bar \xi_t^1,\ldots,\bar \xi^D_t)$, $\bar \Sigma^\xi=\DIAG(\bar \Sigma^{\xi,1},\ldots,\bar \Sigma^{\xi,D})$  and $E(a_t)=a_t$. In general,  $\B$ observation does  not carry  the same information that  zero observation  does. However, when attention is restricted to Gaussian  random variables,  the  conditional  expectation of the state, given zero observation,  is equal to that given  the blank observation  because  the innovation processes associated  with both cases are zero.  Subsequently, from~\cite{meier1967optimal}, one can use the celebrated Kalman filter  to  compute  the optimal state estimate from  noisy observations. In particular, given any realization $a_{1:t}$, the minimization in~\eqref{eq:cost_LQ_optimal} reduces to a mean-square  optimization problem,  where the best nonlinear estimator is known to be $\hat m_{t}=\Exp{m_{t} |o_{1:t},a_{1:t-1}}$.   In a way similar to~\cite{meier1967optimal}, define the  following covariance matrix: 
\begin{align}\label{eq:filter_update}
 P_{t+1}&=  A P_{t} A^\intercal+ \bar \Sigma^w -  A P_{t} C^\intercal(a_t)  ( C(a_t) P_{t}  C^\intercal(a_t) \nonumber \\
 &\quad +  E(a_t) \bar \Sigma^\xi  E^\intercal (a_t))^{-1} C(a_t) P_{t}  A^\intercal, \quad t \in \mathbb{N},
\end{align} 
where $P_{1}=\mathbf{0}_{Dd_s \times Dd_s}$.  Then, the optimal estimator is given by the following Kalman filter:
\begin{equation}
\hat m_{t+1}=   A \hat m_t +L(a_{t})(o_{t+1} - C(a_{t}) A \hat m_t ),
\end{equation}
where $\hat m_1=\mathbf 0_{Dd_s \times 1}$ and the observer gain is described by:
\begin{equation}
L(a_t)=( A P_{t}  C^\intercal (a_t))  (C(a_t) P_{t}  C^\intercal (a_t) + E(a_t)\bar \Sigma^\xi E^\intercal (a_t))^{-1}.
\end{equation}
Consequently, the optimization problem associated with the optimal scheduling strategy  for any finite  horizon $H$  reduces to    a deterministic  non-convex nonlinear  optimization problem as follows:
\begin{equation}
min_{a_{1:H}} \sum_{t=1}^H \gamma^{t-1}(\TR(P_{t+1}) + \ell a_t).
\end{equation}
To find a solution to the above optimization problem, one can construct a dynamic program  based on the history space $\{a_{1:H}\}$,  whose cardinality grows exponentially  with  the horizon (i.e. $2^H$).  Alternatively, one can write a dynamic program  based on the  information state  $P_t$ (which is a continuous variable  in $\mathbb{R}^{Dd_s} \times \mathbb{R}^{Dd_s}$ with nonlinear model-dependent dynamics~\eqref{eq:filter_update}).   For a reasonably large horizon $H$, both dynamic programs  can be very difficult to solve analytically. %The interested reader is referred to~\cite{hespanha2007survey} and references therein for more details.

\section{Generalization to Complex Networks}\label{sec:gen}
The main focus  of the previous sections was to study the trade-off between data collection and data estimation, and  for this reason,  the simplest model structure was considered in order for the excessive number of parameters not to obscure the main  results.  In this section, we show how our  results can  naturally be  extended to  more complex   applications.
\subsection{Multiple  decision makers and estimators} 
Consider a network with $\tilde n \in \mathbb{N}$  decision makers, and  let $\hat n(k) \in \mathbb{N}$   denote the number of  estimators whose access to  data is decided by  decision maker $k \in \mathbb{N}_{\tilde n}$. In such a setup, for any  $j \in \mathbb{N}_{\hat n(k)}$, estimator $j$ provides a different estimate $\hat m^{k,j}$  of the states of all nodes. Thus, the objective  is  to minimize the following social cost function:
\begin{equation}
\mathbb{E}^{g}[\sum_{t=1}^\infty \gamma^{t-1}  \sum_{k=1}^{\tilde n} c^k(m_t,\hat m^{k,1}_t,\ldots,\hat m^{k,\hat n(k)},a^k_t)].
\end{equation}
Since the state dynamics is  not influenced by the actions of  decision makers,  and on the other hand, the above cost function is additive, the optimization problem  of each decision maker  (i.e., Problems~\ref{prob:POMDP} and~\ref{prob:RL}) can be solved  separately, with possibly different parameters.  Therefore, the proposed concept of planning space is  applicable here.
\subsection{Multiple reset  actions}
So far, the trade-off between data collection and data estimation has been formulated as a binary decision, where~$a_t=0$ means that the data are not collected and $a_t=1$ means that the data are collected (note that the collected data are not necessarily credible). It is possible to generalize the above decision to multiple decision options, which correspond to, for example,  using  different routes, channels, sensors and receivers. To this end, it is required to define  the last credible data $x^i_t$ and  elapsed time $y^i_t$ for each option $i \in \mathbb{N}$ so that when  the credible data $x^i_t$ are received at a particular time $t$ via the $i$-th option,  its elapsed time  resets to zero (i.e., $y^i_t=0$)~\cite[Remark 2]{JalalACC2015}. This extension  is similar to a  bandit problem, where each option represents a bandit. Consequently, an immediate application of the data collection/estimation analysis is  to address the trade-off between exploration and exploitation  that arises in various learning tasks, where a decision maker wishes to sequentially choose when to explore (collect the data of interest) and when to exploit the learned model (which is data estimation based on the previously collected data). For multiple reset actions,  see  an example of  machine maintenance problem with three actions in~\cite{Jalal2018fault}.
\subsection{Partially exchangeable and partially equivariant  networks} 
Consider partially exchangeable networks for the  Markov-chain model and partially equivariant  networks for the linear model~\cite{Jalal2019MFT,Jalal2019risk}. The population of nodes is partitioned into a few sub-populations, some with Markov-chain model and some with linear model, as described above, wherein the nodes in the former sub-population are exchangeable and the ones in the latter one are equivariant. In this case, the data dynamics becomes more complex but the  proposed approach still works because the dynamics does not depend on the action of the decision maker. For an example of partially exchangeable network, see a  leader-follower network in~\cite{Jalal2019LCSS} with $n$ exchangeable followers and one non-exchangeable leader.

\subsection{Markovian noise and  credibility processes}
Depending on the data (state of the system), Assumptions~\ref{assum:exchangeable} and~\ref{assum:iid} can be  generalized to the case in which local noises   have their own Markov-chain dynamics. In such a case,   dynamic programming decomposition is still valid, with the only difference  that the state of the Markov chain is  added to the system state.  Note that Assumptions 1 and 2 are not required
for the linear case (i.e., Theorems~\ref{thm:separation} and~\ref{cor:linear}).  Similarly,  credibility can be a Markovian process, e.g., see~\cite{JalalACC2019} for  the spread of fake news in a homogeneous network.

\section{Applications}\label{sec:applications}

\subsection{A sensor network with   prioritized data}\label{sec:sensor}
We  consider a sensor (decision maker) that measures a Markovian source $s_t$ at time $t\in  \mathbb{N},$ such as the temperature of a room or the battery charge state of a smart house,  and  wishes to report it to a data center. Let   $\mathcal{S}:=\mathbb{Z}_{d_s+d_w}$, $d_s,d_w \in \mathbb{N},$ be the state space, and   $s_t$  evolve in time  as follows: 
\begin{equation} \label{eq:dynamics_exp1}
s_{t+1}=\begin{cases}
s_{\max}, &      s_t  >d_s,\\
 s_t+w_t, &  | s_t| \leq d_s,\\
s_{\min}, &      s_t  <-d_s,
\end{cases}
\end{equation}
%$s_{t+1}=s_{\max}$, if    $s_t  >d_s$,
% $s_{t+1}=s_t+w_t$, if $| s_t| \leq d_s$ and 
%$s_{t+1}=s_{\min}$, if       $s_t  <-d_s$,
 where  $s_{\max}, s_{\min} \in \mathcal{S}$  are  the saturation levels,   and   for any $t \in \mathbb{N}$,  $w_t \in \mathcal{W}:=\mathbb{Z}_{d_w}$   is an  i.i.d. process with probability distribution function $P_W$. The state $s_{t+1}$  is successfully received at the data center  upon transmission (i.e., $a_{t}=1$) with probability $q \in [0,1]$ at time $t \in \mathbb{N}$, i.e., for any $s \in \mathcal{S}$,
%\begin{equation}
$ \Prob{o_{t+1}=s | s_{t+1}=s,a_t=1}=q$.
%\end{equation}
In practice, it is not efficient  for the sensor to measure  $s_t$ and transmit it to the data center at each time instant $t \in \mathbb{N}$ because  there is often a cost associated with sensing and transmitting.  Let  $\hat s_t \in \mathcal{S}$ be  the last state received   by time $t \in \mathbb{N}$ at the data center.
%\begin{remark}
%\emph{Notice  that if $\Prob{| s_1+ \sum_{k=1} ^{t-1}w_k| \leq   d_s }\approx 1$ for every $t>1,$ and the noise process is zero-mean,  then $\hat s_t$ is a mean-square estimate. More precisely,  let $\tau  \leq  t$ denote the time instant  at which $\hat s_t$ is received. If   $\tau <t$,  then $\Exp{s_t|o_{1:t}, a_{1:t}} \approx \Exp{ \hat s_\tau +\sum_{k=\tau}^{t-1} w_\tau   \mid x_{1:t},y_{1:t},a_{1:t}}=\hat s_t$, and  if $\tau=t$,  then $\Exp{s_t|o_{1:t}, a_{1:t}}=\hat s_t$ by definition.}
%\end{remark}
 The objective of the sensor  is to find an efficient transmission law that  not only keeps  the estimation error  small at the data center, but  also incurs  minimal  measurement and transmission  cost at the sensor.  The following performance index  is defined:
\begin{equation}\label{eq:cost_exp1}
J=\Exp{\sum_{t=1}^\infty  \gamma^{t-1} \left( |s_t||s_t- \hat s_t|+  \ell a_t \right) }, 
\end{equation}
where $\gamma \in (0,1)$ is  the discount factor and    $\ell \in \mathbb{R}_{>0}$ is the transmission cost.  Note that the estimation error $|s_t- \hat s_t|$ in the above cost function is  weighted by  $|s_t|$;  the rationale for  using such a penalty term  lies in the fact that  in some applications the saturation levels  represent  warning zones,  in the sense  that the estimation error  around  such  zones  carries more weight  than that  around   normal operating  zones: hence,   classical threshold-based strategies that treat  all states equally  are not practical.

\textbf{Example 1.} Suppose  that $s_t$ is the energy level of a battery.  The battery is charged by some renewable generation  sources  and  discharged as serving  demands.   Initially, the nominal value of the battery is $s_1=0$. At each time $t \in \mathbb{N}$,  one  unit energy is added to $s_t$ with probability $p_g$  and  one  unit energy  is depleted from $s_t$ with probability $p_d$,  where the  probability of  the generation $p_g$  is independent of that of the consumption $p_d$. Let $w_t \in \{-1,0,1\}$  denote the change in  the energy level of  the  battery at time $t$, i.e.,  $P(w_t=1)=p_g(1-p_d)$,  $P(w_t=-1)=p_d(1-p_g)$, and $P(w_t=0)=p_g \times p_d+(1-p_g)\times(1-p_d)$.   The objective is to  find  the  optimal  frequency for  transmitting the state of the battery  under the transmission cost $\ell$. Let  $p_g=p_d=0.8$, $d_w=1$, $d_s=s_{\max}=-s_{\min}=99$, $q=0.95$, $\gamma=0.9$, and  $\ell=100$.  Due to the incompleteness of the  information structure  at  the decision making level,  the conventional belief space is large: more precisely,  $\mathbb{P}(s_t\mid o_{1:t}) \in \mathbb{R}^{200}$. In  addition,   reinforcement learning in   belief space is conceptually difficult because  the dynamics of  the belief state depends on the model (i.e., it is a model-dependent planning space). Thus, we  use  a new  information state  based on which the proposed strategies  in  both planning and reinforcement learning cases are  tractable, and their performances are sufficiently close to the optimal performance. The number of  states in the new planning space is $10^3=200 \times 50$.

Let   $x_t \in \mathcal{S}$  denote   the last  credible  observation of the data center  by time $t \in \mathbb{N}$ (i.e., $x_t=\hat s_t$) and $y_t \in \N$  denote  the elapsed time after receiving $x_t$.  A near-optimal transmission law  can be obtained  by solving the Bellman equation~\eqref{eq:Bellman_approximate} in Theorem~\ref{thm:approximate-pomdp} for a  sufficiently large approximation index $k\in \mathbb{N}$, where the space $\Mspace$ and transition probability matrix $\TM$ are replaced by  $\mathcal{S}$ and $T$, respectively, according  to Corollary~\ref{cor:n=1}. After exhaustive simulations, it is observed that the optimal  strategy is obtained  for any approximation index  $k \geq 70$.\footnote{For $\epsilon=10^{-3}$,  inequality~\eqref{eq:k-large} holds for any $k \geq 189$. On the other hand,  it is observed in simulations that the  optimal strategy is obtained for  any $k \geq 70$.} According to~Figure~\ref{fig:sensor},  the  frequency of transmitting the  energy level of the battery to the data center increases as  the energy level approaches  the warning  thresholds.
\begin{figure}[t!]
\centering
\hspace{-.4cm}
\scalebox{1}{
\includegraphics[trim={1.2cm 8cm 0 8.5cm},clip, width=\linewidth]{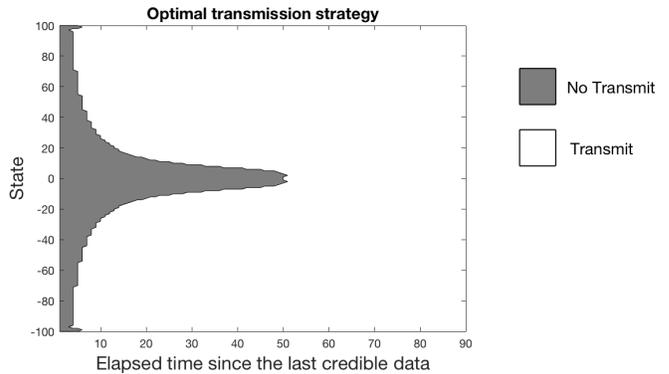}}
%\scalebox{1}{
%\includegraphics[trim={1.2cm 8cm 0 8.5cm},clip, width=\linewidth]{Figures/sensor-known1.pdf}}
\vspace{-.3cm}
\caption{Optimal  transmission strategy in Example 1.}\label{fig:sensor}
\end{figure}
When the generation probability $p_g$, consumption probability $p_d$, and successful delivery probability $q$ are all unknown, one can  use   Q-learning algorithm (Algorithm~\ref{alg.RL-Qlearning}) to obtain the optimal solution. For the purpose of display,  the convergence of the algorithm is depicted at  state $(x,y)=(0,50)$ in Figure~\ref{fig:sensor-unknown}.  It is shown that  $\min_{a \in \mathcal{A}} Q(0,50,a)$ converges to the optimal value function $V(0,50)=160.83$. In this example,  the  Q-learning  algorithm (Algorithm~\ref{alg.RL-Qlearning}) is trained  offline,   where at each training sample a batch update is performed over the entire state-action pairs (Q-functions), also known as synchronized parallel Q-learning, with step sizes inversely proportional to the number of visits (updates) to each pair of state and action.  On a Mac Pro laptop with 2.7 GHz Intel Core i5, the algorithm with a  known-model runs in 250 seconds, and with an unknown-model in  6 hours.  
\begin{figure}[t!]
\centering
\hspace{0cm}
\scalebox{1}{
\includegraphics[trim={0 8.5cm 0 8.5cm},clip, width=\linewidth]{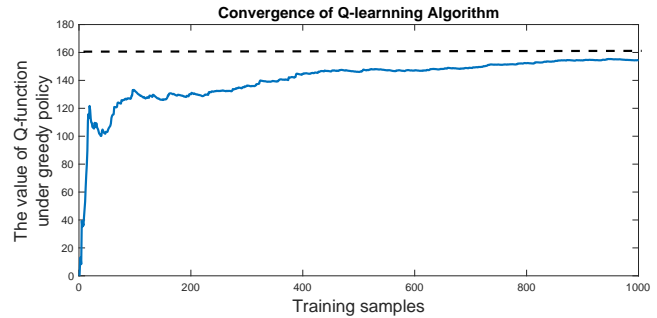}}
\vspace{-.2cm}
\caption{The convergence of  the  Q-learning algorithm (Algorithm~\ref{alg.RL-Qlearning})  at state $(x,y)=(0,50)$ in Example 1.}\label{fig:sensor-unknown}
\end{figure}

\subsection{A communication network with packet drop} 
Inspired by recent developments in  networked control systems~\cite{hespanha2007survey,lipsa2011remote,nayyar2013optimal} and deep teams~\cite{Jalal2019MFT, Jalal2019risk,Jalal2019LCSS},  we consider   $n$ networked controllers  that use  a deep structured optimal state-feedback strategy. In this  case, the dynamics of  the deep state (weighted average of  the states of the controllers) is in the form of~\eqref{eq:dynamic_example_3}.  Now, consider an authority (e.g.,  an independent service operator in a smart grid)  wishing to collect the deep state,  transmit it over an unreliable channel, and eventually receive it at a decoder, located  far away from the control site.  The objective is to find an efficient way  to construct a reliable estimate at the decoder with minimum possible  collection cost, while  taking into account the topology of the network and  unreliability of the transmission (formulated as packet drop).  A block  diagram of the above system is displayed in Figure~\ref{fig:scheme}.

%\edito{At the first step, the dominant features of the network are extracted using the spectral decomposition presented in~Subsection~\ref{sec:linear_model}.  Then, the weighted average of the states (data) corresponding to the largest eigenvalues are computed at the encoder, where it  is decided when to transmit (sample) the  data. The  data are transmitted through an unreliable  communication channel,  where they may or may not be dropped (corrupted) according to a Bernoulli process.  Next, the transmitted data are received at the decoder (that constructs an estimate based on the history of  the received data).  A block  diagram of the above procedure is displayed in Figure~\ref{fig:scheme}.}

\textbf{Example 2.} Consider two topologies, a complete graph and a star graph, with the dominant eigenvalues  $(n-1)$ and $\pm \sqrt{n-1}$, respectively. Suppose $L=1$, $\alpha(0)=0$, and $\alpha(1)=\frac{A}{n-1}$, $A \in \mathbb{R}$, in the vectorized dynamics~\eqref{eq:dynamics_example_3}, where $v^C=\VEC(1,\ldots,1)$ and $v^S=\frac{\sqrt n}{\sqrt{2(n-1)}}\VEC(\pm\sqrt{n-1},1,\ldots,1)$ are the eigenvectors of the dominant eigenvalues of the complete and star graphs, respectively.  Hence,  the dynamics of the weighted average of the dominant mode of the complete graph is given by:
\begin{equation}
m^C_{t+1}=A m^C_{t} +\bar w^C_t,
\end{equation}
where $m^C_t= \frac{1}{n}\sum_{i=1}^n v^c(i)s^i_t=\frac{1}{n}\sum_{i=1}^n s^i_t$ and $\bar w^C_t=\frac{1}{n}\sum_{i=1}^n w^i_t$. Similarly,  the dynamics of the weighted average of the dominant mode of the star graph is described  by:
\begin{equation}
m^S_{t+1}=\frac{\pm A}{\sqrt{n-1}} m^S_{t}+ \bar w^S_t,
\end{equation}
where  $m^S_t=\frac{1}{n}\sum_{i=1}^n v^S(i) s^i_t= \frac{\pm1 }{\sqrt{2n}}s^1_t+\frac{1}{\sqrt{2n(n-1)}}\sum_{i=2}^n s^i_t$  and $\bar w^S_t=\frac{\pm1 }{\sqrt{2n}}w^1_t+\frac{1}{\sqrt{2n(n-1)}}\sum_{i=2}^n w^i_t$,  with   $s^1_t$ and $w^1_t$ denoting the state and local noise  of the  central node.  The per-step cost function is given by~\eqref{eq:cost_LQ_optimal} under Assumption~\ref{ass:limited-LQ}.   In addition, suppose that the  probability of data  being dropped is $1-q=0.1$ and  the discount factor is $ \gamma=0.85$.   Local noises are i.i.d. random variables with  zero mean and  finite variance $\Sigma^w$ (that are not necessarily Gaussian or symmetric with a unimodal
distribution).  From Theorem~\ref{cor:linear}, one can find an optimal strategy for a sufficiently large approximation index $k \in \mathbb{N}$.  The optimal sampling  strategy is shown in Figure~\ref{fig:communication} with respect to the number of nodes $n$  and the variance of  local noises $\Sigma^w$.  In addition, the optimal estimate is constructed at the decoder  based on the result of Theorem~\ref{thm:separation}.   It is observed  from~Figure~\ref{fig:communication} that the data must be sampled more frequently in the complete graph, as the variance of noise increases.  This is not surprising because information flows faster in a complete graph. Furthermore, according to this figure,  the certainty  threshold  of  the complete graph   is larger than  that  of   the star graph,  with respect to  the number of nodes. In these simulations, the approximation index $k$ is  set to $200$, which guarantees $\epsilon$-optimality of the proposed strategies for any $\epsilon \geq 10^{-5}$.
  
In the case when  collection cost $\ell$,  packet-drop probability $1-q$,  number of nodes $n$,  and  system matrices $A$ and $\Sigma^{w}$ are all unknown, one can use  the RL algorithm proposed in~Theorem~\ref{thm:RL} to obtain a near-optimal scheduling  strategy. It is to be noted that the proposed RL is not applicable to the case where the  topology is unknown because we still need to know  the eigenvectors of the graph Laplacian for computing  the data of interest.

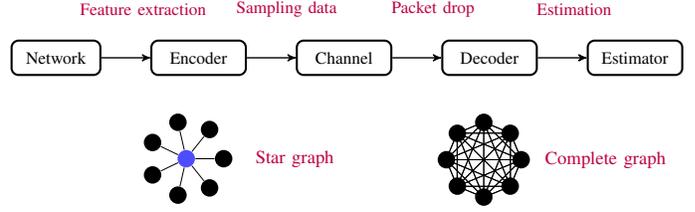
\begin{figure}[t!]
\hspace{-.4cm}
\scalebox{0.65}{
\begin{tikzpicture}[node distance=1cm, auto]
 %nodes
 \node[punkt] (market) {Network};
 \node[punkt, inner sep=5pt,right=1cm of market]
 (encod) {Encoder};
  \node[punkt, inner sep=5pt,right=1cm of encod]
 (chan) {Channel};
   \node[punkt, inner sep=5pt,right=1cm of chan]
 (dec) {Decoder};
  \node[punkt, inner sep=5pt,right=1cm of dec]
 (est) {Estimator};
 % We make a dummy figure to make everything look nice.
 \node[above=.5cm of market] (dummy) {};
 \node[right=of dummy, color=purple] (t) { \vspace*{-.5cm} \hspace{-1cm} Feature extraction};
  \node[above=.5cm of encod] (dummy1) {};
  \node[right=of dummy1, color=purple] (t) {\hspace{-.6cm} Sampling  data};
    \node[above=.5cm of chan] (dummy2) {};
  \node[right=of dummy2, color=purple] (t) {\hspace{-.4cm} Packet drop};
      \node[above=.5cm of dec] (dummy3) {};
  \node[right=of dummy3, color=purple] (t) {\hspace{-.4cm} Estimation};
%   edge[pil,bend left=45] (market.east) % edges are used to connect two nodes
%   edge[pil, bend left=45] (encod.east); % .east since we want
                                             % consistent style
% \node[left=of dummy] (g) {Ultimate lender}
%   edge[pil, bend right=45] (market.west)
%   edge[pil, bend right=45] (formidler.west);
\draw[->,thick](market) -- node{}(encod); 
\draw[->,thick](encod) -- node{}(chan); 
\draw[->,thick](chan) -- node{}(dec); 
\draw[->,thick](dec) -- node{}(est);
\end{tikzpicture}}
\vspace*{.5cm}

\scalebox{0.7}{
\begin{tikzpicture}
      \node[color=purple, thick] (360:0mm)  {\hspace{4cm} Star graph};
    \node[circle,fill=blue!70] at (360:0mm) (center) {};
    \foreach \n in {1,...,7}{
        \node[circle,fill=black!100] at ({\n*360/7}:.7  1cm) (n\n) {};
        \draw (center)--(n\n);
  %      \node at (0,-#2*1.5) {$K_{1,#1}$}; % delete line to remove label
    }
\end{tikzpicture}}

\vspace*{-1.2cm}
\scalebox{0.7}{
\begin{tikzpicture}
\hspace*{5cm}
   \node[color=purple, thick] (360:0mm)  {\hspace{4.5cm} Complete graph};
    \foreach \n in {1,...,8}{
        \node[circle,fill=black!100] at ({\n*360/8}:.7 1cm) (n\n) {}; 
         \foreach \m in {1,...,8}{
        \draw[-, color=black!100 ] ({\n*360/8}:.7 1cm) -- ({\m*360/8}:.7 2cm);    }
  %      \node at (0,-#2*1.5) {$K_{1,#1}$}; % delete line to remove label
    }
\end{tikzpicture}}
\caption{The block diagram of the system in Example 2.}\label{fig:scheme}
\end{figure}

\begin{figure}[t!]
\centering
\hspace{-.6cm}
\scalebox{1}{
\includegraphics[trim={0 5cm 0 6cm},clip, width=\linewidth]{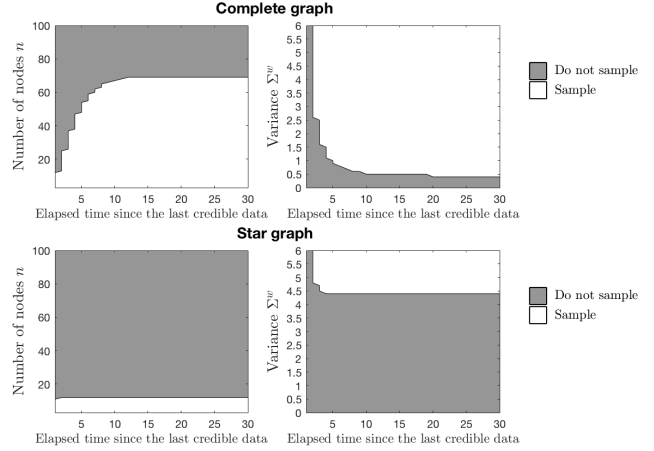}}
\vspace{-.6cm}
\caption{Near-optimal strategies in Example 2.   In the left-hand side plots the dependency of the strategies to the number of nodes and elapsed time is demonstrated, using the parameters:  $A=0.8$,  $\ell=0.4$ and  $\Sigma^{w}=6$.  In the right-hand side plots the dependency of the strategies to the variance of  local noises and elapsed time is displayed,  using parameters: $A=0.9$, $\ell=1$ and  $n=5$. 
}\label{fig:communication}
\end{figure}

\subsection{A social network with binary states}\label{sec:social}
 Consider a social network consisting of $n$  users.  Denote by  $s^i_t \in \mathcal{S}$ the opinion of  user $i \in \mathbb{N}_n$ at time $t \in \mathbb{N}$.  Every  user $i \in \mathbb{N}_n$  independently  changes its opinion  from $s^i_t$ to $s^i_{t+1}$ with probability  $\Prob{s^i_{t+1} \mid s^i_t, m_t}$ at time $t$.  An agency  wants to conduct a survey   from   users to collect the percentage of their opinions and update it frequently;  however,  it  is not practical (and is expensive) to  conduct a survey  at every time step. Consequently, it is desirable to  design a cost-effective yet informative strategy for conducting surveys. 
 
Let  $\ell \in \mathbb{R}_{>0}$ be the price of conducting a survey (e.g., operators' costs and also monetary rewards   to incentivize the users  to  participate) and  $q$ be the probability that  the  result of a survey  is credible. If a survey result is not credible, it is  thrown away.  When there is no survey,  the agency uses  maximum \emph{a posteriori} probability  (MAP) estimator to estimate the percentage of opinions from its previously conducted surveys, i.e.,
\begin{equation}\label{eq:MAP}
\hat m_t=\argmax_{m \in \Mspace} \left(\Prob{m_t=m | o_{1:t}, a_{1:t}}\right).
\end{equation}
Given a discount factor $\gamma \in (0,1)$, the  objective is to   minimize the following  expected total discounted cost:
\begin{equation}
J=\Exp{\sum_{t=1}^\infty \gamma^{t-1} \left( D_{\text{KL}} (m_t || \hat m_t) + \ell a_t  \right) },
\end{equation}
where  $D_{\text{KL}} (m_t || \hat m_t)= \sum_{s \in \mathcal{S}}  m_t(s) \log(\frac{m_t(s)}{\hat m_t(s)}) $ denotes the Kullback–Leibler divergence.   At any time $t \in \mathbb{N}$,  let  $x_t \in \Mspace$  be the  empirical distribution  of the opinions collected at the last  survey and $y_t \in \N$ be the elapsed time after the last survey. From~\eqref{eq:chap-kolm} and~\eqref{eq:MAP},
%\begin{equation}
$\hat m_t=\argmax_{m \in \Mspace} \TM^{y_t}(m,x_t)$,
%\end{equation}
where the transition probability matrix $\TM$ is given by Theorem~\ref{thm:mean_field_iid}.  A near-optimal strategy  can be obtained by Theorem~\ref{thm:approximate-pomdp} for  a sufficiently large $k$. 
%
%  In particular,  for sufficiently large $k \in \mathbb{N}$, define the following Bellman equation: for every $x \in \Mspace$ and $y \in \N_k$ and  any fixed $m^\ast \in \Mspace$,
%\begin{equation}
%V(x,y)=\min(V^0(x,y), V^1(x,y)),
%%\end{equation}
%where $V^0$ is the cost-to-go  corresponding to  no poll, i.e., $a=0$,
%\begin{multline}
%V^0(x,y):= \left[\sum_{m' \in \Mspace} D_{\text{KL}}\left(m' || \argmax_{m \in \Mspace} \TM^{y}(m,x)\right) \TM^y(m',x) \right]\\
%+ \gamma \left[\ID{y < k} V(x,1+y) + \ID{y=k} V(m^\ast,0) \right],
%\end{multline}
%and $V^1$ is the cost-to-go  corresponding to  conducting a poll, i.e., $a=1$, 
%\begin{align*}
%V^1(x,y)&:= \left[\sum_{m' \in \Mspace} \left(D_{\text{KL}}\left(m' || \argmax_{m \in \Mspace} \TM^{y}(m,x)\right)  +\ell \right)\TM^y(m',x) \right]\\
%&\quad + \gamma(1-q) \left[\ID{y < k} V(x,1+y) + \ID{y=k} V(m^\ast,0) \right]\\
%&\quad + \gamma q  \left[\sum_{m' \in \Mspace} \TM^{y+1}(m',x) V(m',0)\right].
%\end{align*}

%\subsection{Numerical example}
\textbf{Example 3.} Consider an election between two candidates $A$ and $B$, i.e., $\mathcal{S}=\{A,B\}$.   An agency  is   interested  to conduct polls among $n \in \mathbb{N}$  voters.  With a  slight abuse of notation,   let   $m_t:=m_t(A) \in \{0, \frac{1}{n},\frac{2}{n}, \ldots,1\}$ represent  the  empirical distribution of the voters  who prefer candidate $A$ at time $t \in \mathbb{N}$.  Since the state space $\mathcal{S}$ is binary, $m_t$ is  sufficient for  determining    the  empirical distribution of the voters  favoring  candidate $B$ which is $1-m_t$.  Note that the empirical distribution $m_t$  takes $n+1$ different values, and its belief state is   $\mathbb{P}(m_t|o_{1:t})  \in \mathbb{R}^{n+1}$. In contrast, our proposed planning space  is a discrete space with $(n+1) \times k$ values, which is a considerable reduction with respect to $n$.  Let the number of voters be $n=50$.  In addition,  let $\Prob{s^i_{t+1}=A \mid s^i_t=A}=0.95$, i.e.,  the probability that voter $i \in \mathbb{N}_n$ chooses candidate $A$  at the next time instant if this is currently the voter's favorite candidate. Similarly, 
$\Prob{s^i_{t+1}=B \mid s^i_t=B}=0.98$, i.e.,  the probability that  voter $i$ chooses candidate $B$ at the next time instant if $B$ is the voter's current choice.  Suppose that the cost of running a poll is   $\ell=0.02$, the discount factor is $\gamma=0.8$, and the  probability of the credibility of a poll is $q=0.95$.  Exhaustive simulations demonstrate that the optimal solution is attained for any approximation index $k \geq 50 $.\footnote{For $\epsilon=10^{-5}$,  inequality~\eqref{eq:k-large} holds for any $k \geq 50$.}  The optimal strategy is displayed in Figure~\ref{fig:election-known}, which demonstrates that when the number of voters in favor of candidate $A$ is $45$, the agency should run a poll after the elapsed time from the latest credible observation exceeds $10$. When, on the other hand,  the number of voters in favor of candidate $B$ is $45$, the agency should wait slightly longer ($11$ time instants) before conducting a new poll. This difference is due to the fact that  voters  are more likely to change their opinion if candidate $A$ is their current choice.     The simulation time is  $116$ seconds on a computer with the specification described   in Example~1.
\begin{figure}[t!]
\centering
\scalebox{1}{
\includegraphics[trim={0 9cm 0 9cm},clip, width=\linewidth]{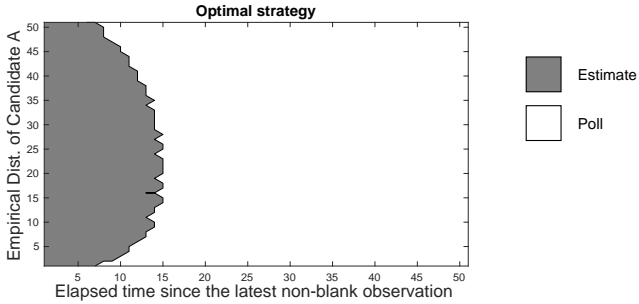}} 
\vspace{-.3cm}
\caption{Optimal strategy  for conducting polls in Example 3. The unsmooth surface of the solution  is  due  to the fact  that the MAP estimator is not smooth.}\label{fig:election-known}
\end{figure}

%For a sufficiently large $n$,  one can use the mean-field approximation given by~\eqref{eq:h_large_T}   along with  a simple quantization technique~\cite{JalalCDC2017} to find a tractable (scale-free) sub-optimal solution.

%If the transition probabilities as well as the probability of the credibility of polls are unknown, one may use Theorem~\ref{thm:machine} to find a near-optimal solution.

%\section*{Acknowledgement}
%This work has been supported in part by the Natural Sciences and Engineering Research Council of Canada (NSERC) under Grant RGPIN-262127-17, and in part by Concordia University under Horizon Postdoctoral Fellowship.
\section{Conclusions}\label{sec:conclusions}
In this paper,   the trade-off between   data collection   and estimation  in  networks with  both known and unknown models was investigated.  Some  important properties of the dynamics of data were studied first, and  an $\varepsilon$-optimal solution was subsequently provided using the  Bellman equation.  The proposed solution was then extended to the case where the  model is not completely known, using  two different  learning-based approaches. It was also shown at what point   estimating data tends to be more desirable than collecting data as  the number of nodes  increases. The special case of linear dynamics was studied  in more detail, and a separation principle was presented accordingly. Three numerical examples were   provided to demonstrate the effectiveness  of the proposed strategies.

For future research directions, it would be  interesting  to study the  computational complexity of the proposed approach  under various  approximation methods such as stochastic approximation, linearization, particle filtering, quantization, randomization,   and  Monte-Carlo simulation  as well as  different simplifying, yet realistic, assumptions  such as  ergodicity and the myopicity of the decision process.  In particular, the reader is referred to~\cite{toral2014stochastic} and references therein for various stochastic  numerical  methods that can be used for more efficient  computation, albeit at the cost of losing the performance guarantee.

\printbibliography
%\bibliographystyle{IEEEtran}
%\bibliography{Jalal_Ref}

\appendices

\section{Proof of Theorem~\ref{thm:mean_field_iid}}\label{sec:proof_thm:mean_field_iid}
From~\eqref{eq:mf-def} and~\eqref{eq:model-dynamics}, it follows that for any $s' \in \mathcal{S},t \in \mathbb{N}$,  
\begin{align}\label{eq:proof_tm_iid}
n \cdot m_{t+1}(s')&=\sum_{i=1}^n \ID{s^i_{t+1}=s'}=\sum_{i=1}^n \ID{f(s^i_t,m_t,w^i_t)=s'} \nonumber \\
&= \sum_{s \in \mathcal{S}} \sum_{i=1}^n \ID{s^i_t=s} \ID{f(s,m_t,w^i_t)=s'}.
\end{align}
For every state $s \in \mathcal{S}$, the inner sum in  equation~\eqref{eq:proof_tm_iid}  has $n(1- m_t(s))$ zero terms because there are only $n m_t(s)$ terms with state $s$ according to the definition of empirical distribution~\eqref{eq:mf-def}. These $nm_t(s)$ possibly non-zero terms are independent  binary random variables,  according to Assumption~\ref{assum:iid}, with the success probability
%\begin{align}
$\Prob{\ID{f(s,m_t,w^i_t)=s'}=1}=\sum_{w \in \mathcal{W}} P_W(w^i_t=w)
\times \ID{f(s,m_t,w)=s'} =T(s',s,m_t)$.
%\end{align}
Let $\phi_{m(s)}(s',s,m_t)$ denote the probability distribution function of the sum of these $nm_t(s)$ Bernoulli random variables, which is a binomial distribution with  $nm_t(s)$ trials and success probability $T(s',s,m_t)$. Now, the probability distribution function of $nm_{t+1}(s')$ is  the probability distribution function of the outer sum in equation~\eqref{eq:proof_tm_iid}, consisting of  $|\mathcal{S}|$  independent random variables  (due to  Assumption~\ref{assum:iid}), each of which has the probability distribution function $\phi_{m(s)}(s',s,m_t)$. Therefore, the probability distribution of $nm_{t+1}(s')$  can be expressed as the convolution of the probability distribution functions $\phi_{m(s)}(s',s,m_t)$ over space $\mathcal{S}$, denoted by $\bar \phi(s',m_t)$.

\section{Proof of Lemma~\ref{lemma:tansition_prob_o}}\label{sec:proof_lemma:tansition_prob_o}
One can write:
 \begin{multline}\label{eq:proof-o-t+1}
 \Prob{o_{t+1} \mid x_{1:t}, y_{1:t}, a_{1:t}}   =\sum_{m'}\Prob{m_{t+1}=m' \mid x_{1:t}, y_{1:t}, a_{1:t}}\\
 \times  \Prob{o_{t+1}\mid m_{t+1}=m', x_{1:t}, y_{1:t}, a_{1:t}}.
 \end{multline}
From~\eqref{eq:observation_blank} and~\eqref{eq:observation_not_blank}, the multiplicand  in the right-hand side of equation~\eqref{eq:proof-o-t+1} is given by:
 \begin{multline}\label{eq:proof-o-t+1-first}
 \Prob{o_{t+1}\mid m_{t+1}=m', x_{1:t}, y_{1:t},  a_{1:t}}= a_tq \ID{o_{t+1}=m'} \\
  + (1-a_t q) \ID{o_{t+1}=\B}.
 \end{multline}
 The multiplier  in the right-hand side of equation~\eqref{eq:proof-o-t+1}, on the other hand,  is:
 \begin{align} \label{eq:proof-o-t+1-second}
 &\Prob{m_{t+1}=m' \mid x_{1:t}, y_{1:t},  a_{1:t}} \nonumber \\
&= \sum_{m \in \Mspace} \Prob{m_{t+1}=m' \mid m_t=m,  x_{1:t}, y_{1:t}, a_{1:t}} \times \nonumber \\
&\qquad  \Prob{m_{t}=m \mid   x_{1:t}, y_{1:t},  a_{1:t}} \nonumber \\
&\substack{(a)\\=} \sum_{m \in \Mspace} \Prob{m_{t+1}=m' \mid m_t=m} \times  \nonumber\\ 
&  \frac{\ID{a_t=g_t(o_{1:t},a_{1:t-1})} \Prob{m_{t}=m \mid   x_{1:t}, y_{1:t}, a_{1:t-1}}}{\sum_{\tilde m } \ID{a_t=g_t(o_{1:t},a_{1:t-1})} \Prob{m_t=\tilde m \mid   x_{1:t}, y_{1:t},  a_{1:t-1}}}  \nonumber \\
 &\substack{(b)\\=}  \sum_{m \in \Mspace}  \TM(m',m)\cdot \TM^{y_t}(m,x_t)=\TM^{y_t+1}(o_{t+1},x_t),
 \end{align} 
 where $(a)$ follows from Proposition~\ref{thm:mean_field_evolution_exchangeable}, the fact that  $\mathbf w_t$ is independent of  the information up to   time $t$, i.e. $(o_{1:t}, a_{1:t}, x_{1:t}, y_{1:t})$, as well as  Bayes' rule, and (b) follows from~\eqref{eq:tm} and~\eqref{eq:chap-kolm}. The proof follows from~\eqref{eq:proof-o-t+1}--\eqref{eq:proof-o-t+1-second}.
 
 \section{Proof of Lemma~\ref{lemma:hat c}}\label{sec:proof_lemma:hat c}
To prove the lemma, it is noted that by definition:
 \begin{align}
  &\Exp{c(m_t,\hat m_t, a_t) \mid o_{1:t}, a_{1:t}}  \\
   &= \sum_{m, \hat m \in \Mspace}    c(m,\hat m, a_t) \Prob{m_t=m, \hat m_t=\hat m \mid o_{1:t}, a_{1:t}} \nonumber \\
  &\substack{(a) \\=} \sum_{m, \hat m \in \Mspace}   \hspace{-.3cm} c(m,\hat m, a_t) \ID{\hat m=h(\Prob{m_t=m \mid o_{1:t}, a_{1:t-1}})} \nonumber \\
   &\times \frac{\ID{a_t=g_t(o_{1:t},a_{1:t-1})}  \Prob{m_t=m\mid o_{1:t}, a_{1:t-1}} }{\sum_{\tilde m \in \Mspace} \ID{a_t=g_t(o_{1:t},a_{1:t-1})}  \Prob{m_t=\tilde m\mid o_{1:t}, a_{1:t-1}}}  \nonumber \\
 & \substack{(b)\\=} \sum_{m, \hat m \in \Mspace} c(m,\hat m, a_t) \ID{\hat{m}=h(\TM^{y_t}(m,x_t))}  \TM^{y_t}(m,x_t), 
 \end{align}
 where $(a)$ follows from \eqref{eq:estimator:general} as well as Bayes' rule, and $(b)$ follows from~\eqref{eq:chap-kolm}.
 
 \section{Proof of Theorem~\ref{thm:approximate-pomdp}}\label{sec:proof_thm:approximate-pomdp}
 The proof consists of  two parts. In the first part, we define a ``virtual" finite-state Markov decision process (MDP), and in the second part, we show  that the optimal solution of this MDP is an $\varepsilon$-optimal solution of Problem~\ref{prob:POMDP}.

\textbf{Part 1:}
For any $ t \in \mathbb{N}$ and finite $k \in \mathbb{N}$, define a so called ``virtual" finite-state MDP with state  $(\tilde x_t, \tilde y_t) \in \Mspace \times \N_k$ and   action $\tilde a_t \in \mathcal{A}$, as well as the initial state $(\tilde x_1,\tilde y_1)=(x_1,y_1)=(m_1,0)$. At time $t \in \mathbb{N}$,  state $(\tilde x_t, \tilde y_t) $  evolves according to function  $\tilde  f: \Mspace \times \N_k  \times \mathcal{O} \rightarrow   \Mspace \times  \N_k $ as
%\begin{equation}
$(\tilde x_{t+1}, \tilde y_{t+1})=\tilde f(\tilde x_{t}, \tilde y_{t}, \tilde  o_{t+1})$
%\end{equation}
such that
\begin{equation}\label{eq:tilde-f}
\tilde f(\tilde x_{t}, \tilde y_{t}, \tilde  o_{t+1}):=\begin{cases}
(\tilde x_t, \tilde y_t+1), & \tilde o_{t+1}=\B, \tilde y_{t} <k,\\
(m^\ast,0), & \tilde  o_{t+1}=\B,\tilde y_t= k,\\
(\tilde o_{t+1},0), & \tilde o_{t+1} \neq \B,
\end{cases}
\end{equation}
where $\tilde o_{t+1} \in \mathcal{O}$ is a noise process and $m^\ast \in \Mspace$ is an arbitrary empirical distribution. The  probability distribution of the noise $\tilde o_{t+1}$  is identical to that of the observation in the original model~\eqref{eq:tansition_prob_o}, i.e.,
%\begin{multline}\label{eq:tansition_prob_o_tilde}
$\mathbb{P}  \left( \tilde o_{t+1} \mid \tilde x_{1:t},\tilde y_{1:t},\tilde a_{1:t}\right)=  (1-\tilde a_tq) \ID{\tilde o_{t+1}=\B}$ 
$+\tilde a_t q      \TM^{\tilde y_t+1}(\tilde o_{t+1},\tilde x_t) \ID{\tilde o_{t+1} \neq \B}$.
%\end{multline}
The per-step cost of the virtual model  introduced above  is the restriction function $\hat c$ given by~\eqref{eq:hat-c}, over space  $\Mspace \times \N_k \times \mathcal{A}$, i.e., at time $t \in \mathbb{N}$,
%\begin{equation}\label{eq:cost-virtual-tilde}
$\hat c(\tilde x_t, \tilde y_t, \tilde a_t):= \sum_{m \in \Mspace} c(m,h(\TM^{\tilde y_t}(m,\tilde x_t)), a_t)    \TM^{\tilde y_t}(m, \tilde x_t)$.
%\end{equation}
The strategy of  the virtual model  is given by
%\begin{equation}
$\tilde a_t=\tilde g_t(\tilde x_{1:t}, \tilde y_{1:t})$,
%\end{equation}
 and   its  performance is described by
%\begin{equation}\label{eq:tilde_J}
$\tilde J(\tilde g)=\mathbb{E}^{\tilde g}[\sum_{t=1}^\infty \gamma^{t-1} \hat c(\tilde x_t,\tilde y_t, \tilde a_t)]$.
%\end{equation}
From the standard results in Markov  decision theory~\cite{Bertsekas2012book}, the optimal solution of the virtual model is obtained  by solving  the following Bellman equation for any $\tilde x  \in \Mspace, \tilde y \in \N_k$:
\begin{equation}\label{eq:bellman_infinite_tilde}
\tilde V_k(\tilde x, \tilde y)=\min_{\tilde a \in \mathcal{A}} ( \hat c(\tilde x,\tilde y,\tilde a) + \gamma \Exp{\tilde V_k(\tilde  f(\tilde x,\tilde y,\tilde o))} ),
\end{equation}
where the above expectation is taken over all noises $\tilde o \in \mathcal{O}$.

\textbf{Part 2:} Let $J^\ast$ be the performance under the optimal solution of Bellman equation~\eqref{eq:bellman_infinite} and $\tilde J^\ast $ be the  performance  under the optimal  solution of Bellman equation~\eqref{eq:bellman_infinite_tilde}. We compute an upper bound on the relative distance between $J^\ast$ and  $\tilde J^\ast $, i.e. $|J^\ast - \tilde J^\ast|$,  as follows:
\begin{align}\label{eq:inequality-proof-4}
&   |  \min_{g} \mathbb{E}^{g}\sum_{t=1}^\infty \gamma^{t-1}\hat c(x_t,y_t,a_t) \hspace{-.05cm}- \hspace{-.05cm} \min_{\tilde g} \mathbb{E}^{\tilde g}\sum_{t=1}^\infty \gamma^{t-1}\hat  c(\tilde x_t,\tilde y_t,\tilde a_t) |  \nonumber \\
&= |  \min_{g} \mathbb{E}^{g}[\sum_{t=1}^k \gamma^{t-1}\hat c(x_t,y_t,a_t) +\sum_{t=k+1}^\infty \gamma^{t-1}\hat c(x_t,y_t,a_t)] \nonumber \\
&\quad  - \min_{\tilde g} \mathbb{E}^{\tilde g}[\sum_{t=1}^k \gamma^{t-1}\hat  c(\tilde x_t,\tilde y_t,\tilde a_t) - \sum_{t=k+1}^\infty \gamma^{t-1}\hat  c(\tilde x_t,\tilde y_t,\tilde a_t)]|  \nonumber \\
&\substack{(a)\\ \leq}     \frac{2 \gamma^k  c_{\text{max}} }{1- \gamma},
\end{align}
where $(a)$ follows from  the triangle inequality, the fact that $c_{\text{max}}$ is an upper bound for the per-step cost,  and that the minimization of the expected cost of the original model and that of the virtual model up to  time $k$ are essentially  the same,  as  both  models  start from  the same  initial state $(m_1,0)$, follow the same dynamics, and incur the same cost up to time~$k$.  Finally, it is concluded from Parts 1 and 2 that  when $k$ is sufficiently large such that  $\frac{2 \gamma^k  c_{\text{max}} }{1- \gamma} \leq \varepsilon$, the optimal solution of the Bellman equation~\eqref{eq:bellman_infinite_tilde} is an $\varepsilon$-optimal solution of Problem~\ref{prob:POMDP}.  The proof is completed by incorporating equation~\eqref{eq:tilde-f} in the Bellman equation~\eqref{eq:bellman_infinite_tilde}.

\section{Proof of Theorem~\ref{thm:RL}}\label{sec:proof_thm:RL}
 Let $\tilde V^\ast_k: \Mspace \times \N_k \rightarrow \mathbb{R}_{\geq 0}$ be the optimal value function   satisfying~\eqref{eq:bellman_infinite_tilde}. Define function $Q^\ast: \Mspace \times \N_k \times \mathcal{A} \rightarrow \mathbb{R}_{\geq 0}$ such that for every $x \in \Mspace, y \in \N_k, a \in \mathcal{A}$,
\begin{equation}\label{eq:Q_ast}
Q^\ast(x,y,a):= \hat c(x,y,a)+ \gamma \sum_{o \in \mathcal{O}} \Prob{o|x,y,a} \tilde V^\ast_k(\tilde f(x,y,o)).
\end{equation}
It follows from equations~\eqref{eq:bellman_infinite_tilde} and~\eqref{eq:Q_ast} that  $\tilde V_k^\ast(x,y)=\min_{a \in \mathcal{A}} Q^\ast(x,y,a)$ for every  $ x \in \Mspace$ and $y \in \N_k$, i.e.,
\begin{equation}\label{eq:Q_ast_1}
\Compress
Q^\ast(x,y,a)= \hat c(x,y,a)+ \gamma \sum_{o \in \mathcal{O}} \Prob{o|x,y,a} \min_{a' \in \mathcal{A}} Q^\ast(\tilde f(x,y,o),a').
\end{equation}
According to  Theorem~\ref{thm:approximate-pomdp} and equations~\eqref{eq:bellman_infinite_tilde} and~\eqref{eq:Q_ast}, any $\argmin_{a \in \mathcal{A}} Q^\ast(x,y,a)$  is an argmin for the Bellman equation~\eqref{eq:Bellman_approximate}. Now, rewrite equation~\eqref{eq:Q_ast_1} as 
%\begin{equation}\label{eq:Q_ast_2}
$Q^\ast(x,y,a)= c'  +   \gamma  \min_{a' \in \mathcal{A}} Q^\ast(\tilde f(x,y,o),a')+ n(x,y,o)$,
%\end{equation}
where $c'$ denotes  the instantaneous cost at state ($x,y$) and action $a$, and  the random variable $n(x,y,o)$ is defined by
\begin{multline}
n(x,y,o):=  -c' +\hat c(x,y,a)  - \gamma  \min_{a' \in \mathcal{A}} Q^\ast(\tilde f(x,y,o),a') \\
+\gamma \sum_{o \in \mathcal{O}} \Prob{o|x,y,a} \min_{a' \in \mathcal{A}} Q^\ast(\tilde f(x,y,o),a'),
\end{multline}
such that $\mathbb{E}[n(x,y,o)\mid x,y,a]=0$, and
%\begin{equation}
$\mathbb{E}[ n(x,y,o)^2\mid x,y,a] \leq  4 c_{\text{max}}^2 + 4 (\max_{x',y', a'} Q^\ast((x',y'),a'))^2$,
%\end{equation}
where  $n(x,y,o) \leq 2c_{\text{max}} +2\max_{x',y', a'} Q^\ast((x',y'),a')$. The decision maker can use stochastic approximation theory to approximate the $Q^\ast$-function described by equation~\eqref{eq:Q_ast_1}, because: (i) function $\tilde f$ is independent of the model;  (ii) the expectation and variance of $n(x,y,o)$ are respectively  zero and finite, and (iii) equation~\eqref{eq:Q_ast_1}  is a contraction mapping  in the infinity norm due to the discount factor $\gamma<1$, i.e., for  any  $Q$ and $Q'$:
\begin{equation}
\Ninf{F(Q)- F(Q')} \leq \gamma \max_{x',y',a'} |Q((x',y'),a')    - Q'((x',y'),a') |,
\end{equation}
where $F$ denotes the function form of equation~\eqref{eq:Q_ast_1} such that $Q^\ast=F(Q^\ast)$. Therefore,   the following  stochastic approximation iteration  converges to $Q^\ast$ under standard assumptions in \cite[Theorem 4]{Tsitsiklis1994asychronous}, i.e., for  $\tau \in \mathbb{N}$,
\begin{multline}
Q_{\tau+1}(x,y,a)= Q_\tau(x,y,a)+ \\
 \alpha_\tau (x,y,a)  \big( c'+\gamma \min_{a' \in \mathcal{A}} Q_\tau(\tilde f(x,y,o),a') 
- Q_\tau(x,y,a) \big).
\end{multline}
   The proof  is completed on noting that the obtained  greedy strategy $g^\ast_\varepsilon$  is an $\varepsilon$-optimal solution,  according to  Theorem \ref{thm:approximate-pomdp}.
 
\section{Proof of Theorem~\ref{thm:asymptotic}}\label{sec:proof_thm:asymptotic}
 The total expected discounted cost for any finite horizon $H \in \mathbb{N}$ under no collection action  is given by:
 \begin{align}\label{eq:proof_long}
 &\Exp{\sum_{t=1}^H \gamma^{t-1} c(m_t,\hat m_t,0)}   \substack{(a) \\ \leq } \Exp{\sum_{t=1}^H   K_c \gamma^{t-1} \Ninf{m_t - \hat m_t} } \nonumber \\
& \substack{(b) \\ \leq } K_c \Exp{\Ninf{m_1 -\hat m_1}}+
 K_c \sum_{t=2}^H  \gamma^{t-1}  \mathbb{E} \Big[ K_p^{t-1} \Ninf{m_1 - \hat m_1}\nonumber \\
 & \quad + (\sum_{\tau=1}^{t-1} K_p^{\tau-1}) \mathcal{O}(\frac{1}{\sqrt n}) \Big]\nonumber \\
 & \substack{(c) \\ = } K_c \sum_{t=2}^H  \gamma^{t-1}  (\sum_{\tau=1}^{t-1}  K_p^{\tau-1}) \mathcal{O}(\frac{1}{\sqrt n}) \nonumber \\
 & \substack{(d) \\ = } K_c \sum_{t=2}^H  \gamma^{H-t+1}  (\sum_{\tau=1}^{t-1}   (\gamma K_p)^{\tau-1}) \mathcal{O}(\frac{1}{\sqrt n}) \nonumber \\
  & \substack{(e) \\ \leq } K_c \sum_{t=2}^H  \gamma^{H-t+1}  (\sum_{\tau=1}^{H-1}   (\gamma K_p)^{\tau-1}) \mathcal{O}(\frac{1}{\sqrt n}) \nonumber \\
   & = K_c (\sum_{t=2}^H  \gamma^{H-t+1} ) (\sum_{\tau=1}^{H-1}   (\gamma K_p)^{\tau-1}) \mathcal{O}(\frac{1}{\sqrt n}) \nonumber \\
      & \substack{(f) \\ = } K_c \times \frac{\gamma(1- \gamma^{H-1})}{1-\gamma} \times \frac{1-(\gamma K_p)^{H-1}}{1-\gamma K_p} \times \mathcal{O}({\frac{1}{\sqrt n}}),
 \end{align}
 where $(a)$  follows from Assumption~\ref{assum:Lischitz} and the monotonicity of the expectation operator; $(b)$ follows from Lemma~\ref{lemma:evolution_mean_field} (by applying it recursively); $(c)$  follows  from the fact that $m_1=\hat m_1$; $(d)$  rearranges the terms; $(e)$ follows from the fact that $\gamma K_p$ is non-negative  along with the inequality $t \leq H$, and  $(f)$  follows from Assumption~\ref{assump:gamma_k}. The proof is completed by tending  horizon $H$ to~$\infty$.
 
 \section{Proof of Theorem~\ref{thm:separation}}\label{sec:proof_thm:separation}
 Given any strategy $g$,  the  information set $\{o_{1:t},a_{1:t-1}\}$ can be equivalently expressed by the set $\{x_{1:t},y_{1:t}\}$, according to~\eqref{eq:strategy} and \eqref{eq:o-x-y}. Hence,  the generic estimator $h$ can be represented by a strategy-dependent estimator  $\hat m_t=h^g_t(x_{1:t},y_{1:t})$. From~\eqref{eq:dynamic_example_3} and the definition of  the set $\{x_{1:t},y_{1:t}\}$, one has:
\begin{equation}\label{eq:proof-separation1}
m_t= A^{y_t} x_t + \ID{y_t > 0} \sum_{\tau=1}^{y_t}  A^{\tau-1} \bar w_{t-\tau}.
\end{equation}
  Since local  noises from  time $t-y_t$ to time $t$  have zero mean, and  are mutually independent over  time (and so is $\bar w_{t-y_t:t}$),  the per-step cost function can be described as follows:
\begin{align}
&\Exp{(m_t - \hat m_t)^\intercal (m_t - \hat m_t)z(a_t) + \ell (m_t,a_t) \mid x_{1:t}, y_{1:t},g_{1:t}}\\
&=\Exp{(A^{y_t} x_t - h_t^g(x_{1:t},y_{1:t}))^\intercal   (A^{y_t} x_t - h_t^g(x_{1:t},y_{1:t})} \\
&\times z(g_t(x_{1:t},y_{1:t})) \\
&+ \ID{y_t > 0}  z(g_t(x_{1:t},y_{1:t})) \sum_{\tau=1}^{y_t}  \Exp{(\bar{w}_{t-\tau})^\intercal (A^{\tau-1})^\intercal A^{\tau-1}  \bar{w}_{t-\tau}} \\
&+ \Exp{\ell( A^{y_t} x_t +\ID{y_t > 0}  \sum_{\tau=1}^{y_t}   A^{\tau-1} \bar w_{t-\tau}, g_t(x_{1:t},y_{1:t}))}.
\end{align}
Therefore, for   any control  horizon $H$,  any sample path  $\{x_{1:H},y_{1:H}\}$,  and any strategy $g_{1:H}$,   the first term in the right-hand side of the above equation is the only term that is affected by the choice of $h^g$,  which yields  the unique minimizer $h^g_t(x_{1:t},y_{1:t})= A^{y_t} x_t$. Note  that the structure of this estimator  is  independent of  the probability distribution of the underlying random variables, strategy $g$, and the order of dynamics, and follows the  update rule $\hat m_{t+1}= A^{y_{t+1}} x_{t+1}$, i.e., equation~\eqref{eq:update_rule_KF}.
This estimator  has the same structure as  the minimum  mean-square  estimator $\Exp{m_t \mid o_{1:t},a_{1:t-1}}$. Hence, 
\begin{equation}\label{eq:update_rule_KF-2}
\hat m_{t+1}=\begin{cases}
x_{t+1}=o_{t+1}, &  o_{t+1} \neq \B \text{ (i.e., $y_{t+1}=0$)},\\
A^{1+y_{t}} x_{t}=A \hat m_t, & o_{t+1} = \B \text{ (i.e., $y_{t+1}> 0$)}.
\end{cases}
\end{equation}

\begin{IEEEbiography}[{\includegraphics[width=1in,height=1.25in,clip,keepaspectratio]{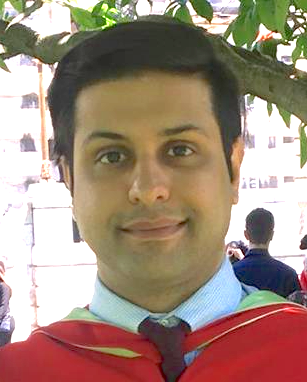}}]{Jalal Arabneydi}
received the Ph.D. degree  in Electrical and Computer Engineering from
  McGill University, Montreal, Canada in 2017.  
  He is currently a postdoctoral
  fellow at Concordia University. He  was  the recipient of the
best student paper award at the 53rd Conference on
Decision and Control (CDC), 2014. His  principal research interests
  include  stochastic control,  robust optimization, game theory,  large-scale system, multi-agent reinforcement learning  with applications in complex networks including  smart grids, swarm robotics,  and finance.  His current research interest is  focused on what he calls deep planning, which bridges  decision making theory and artificial intelligence. The ultimate goal is to define proper mathematical tools and solution concepts in order to develop large-scale decision-making algorithms that work under imperfect information and incomplete knowledge with analytical performance guarantees.
\end{IEEEbiography}
\begin{IEEEbiography}
[{\includegraphics[width=1in,height=1.25in,clip,keepaspectratio]{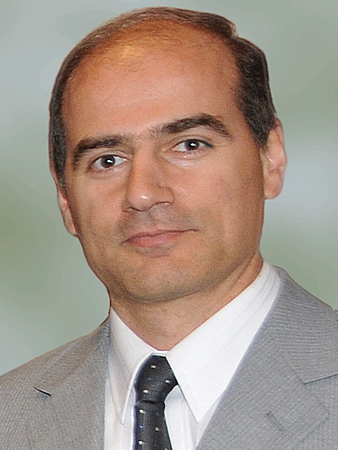}}]{Amir G. Aghdam}
 received the Ph.D. degree in electrical and computer engineering from the University of Toronto, Toronto, ON, Canada, in 2000. He is currently a Professor in the Department of Electrical and Computer Engineering at Concordia University, Montreal, QC, Canada. He was a Visiting Scholar at Harvard University in fall 2015, and was an Associate with the Harvard John A. Paulson School of Engineering and Applied Sciences from September 2015 to
December 2016. His research interests include multi-agent networks, distributed control, optimization and sampled-data systems.  He is a member of Professional Engineers Ontario, chair of the Conference Editorial Board of IEEE Control Systems Society,  Editor-in-Chief of the IEEE Systems Journal, and was an Associate Editor of the IEEE Transactions on Control Systems Technology and the Canadian Journal of Electrical and Computer Engineering. He has been a member of the Technical Program Committee of a number of conferences, including the IEEE Conference on Systems, Man and Cybernetics (IEEE SMC)Inline image
 and the IEEE Multiconference on Systems and Control (IEEE MSC). He was a member of the Review Committee for the Italian Research and University Evaluation Agency (ANVUR) for 2012–2013, and a member of the Natural Sciences and Engineering Research Council of Canada (NSERC) ECE Evaluation Group for 2014–2016. He is a recipient of the 2009 IEEE MGA Achievement Award, the 2011 IEEE Canada J. J. Archambault Eastern Canada Merit Award, and  the 2020 IEEE Canada J. M. Ham Outstanding Engineering Educator Award. He was the 2014–2015 President of IEEE Canada and Director (Region 7), IEEE, Inc., and was also a member of the IEEE Awards Board for this period. Dr. Aghdam was a member of the IEEE Medal of Honor Committee for 2017-2019, and IEEE MGA Awards and Recognition Committee for 2017-2018, and is currently the Vice-Chair of the IEEE Medals Council.
\end{IEEEbiography}

\end{document}